\documentclass[12pt]{article}

\usepackage{amsmath}
\usepackage{amsthm}
\usepackage{amssymb}
\usepackage{cite}
\usepackage{url}
\usepackage[multiple]{footmisc}
\usepackage{graphicx}
\usepackage{epstopdf}
\usepackage{chngcntr}
\usepackage{enumitem}
\usepackage{hhline}
\usepackage{array}
\usepackage{hyperref}
\usepackage{color}
\usepackage{array}
\usepackage{multirow}

\hypersetup{colorlinks=false}

\if@twoside
\else 
\fi

\numberwithin{equation}{section}
\counterwithin{figure}{section}
\counterwithin{table}{section}

\theoremstyle{plain}% default
\newtheorem{theorem}{Theorem}[section]
\newtheorem{lemma}[theorem]{Lemma}

\theoremstyle{definition}
\newtheorem{definition}{Definition}[section]
\newtheorem{assumption}{Assumption}[section]

\theoremstyle{remark}
\newtheorem{remark}{Remark}[section]

\newcommand{\norm}[1]{\left\|#1\right\|}
\newcommand{\abs}[1]{\left\vert#1\right\vert}
\newcommand{\spr}[1]{\left\langle\,#1\,\right\rangle}
\newcommand{\kl}[1]{\left(#1\right)}
\newcommand{\Kl}[1]{\left\{#1\right\}}

\newcommand{\supp}[1]{\text{supp}\left(#1\right)}

\definecolor{aog}{rgb}{0.0, 0.5, 0.0}

\newcommand{\LandauO}{\mathcal{O}}

\newcommand{\R}{\mathbb{R}} 
\newcommand{\C}{\mathbb{C}}
\newcommand{\N}{\mathbb{N}}
\newcommand{\Z}{\mathbb{Z}}

\newcommand{\xD}{x^\dagger}

\newcommand{\xad}{x_\alpha^\delta}

\newcommand{\yd}{y^{\delta}}

\newcommand{\za}{z_\alpha}
\newcommand{\zad}{z_\alpha^\delta}

\newcommand{\eps}{\varepsilon}

\newcommand{\vphi}{\varphi}

\newcommand{\et}{\tilde{e}}
\newcommand{\lt}{{\ell_2}}

\newcommand{\AD}{\mathcal{A}}

\newcommand{\Ft}{\tilde{F}}

\newcommand{\fkt}{\tilde{f}_k}
\newcommand{\ekt}{\tilde{e}_k}

\newcommand{\ak}{\alpha_k}

\newcommand{\OD}{\Omega_D}
\newcommand{\OS}{\Omega_S}
\newcommand{\OSb}{\Omega_S'}

\newcommand{\LtR}{{L_2(\R)}}
\newcommand{\LtT}{{L_2(0,2\pi)}}
\newcommand{\LtOD}{{L_2(\Omega_D)}}
\newcommand{\LtOS}{{L_2(\Omega_S)}}
\newcommand{\LtOSb}{{L_2(\OSb)}}

\newcommand{\ejkl}{e_{j,k,l}}
\newcommand{\fjkl}{f_{j,k,l}}
\newcommand{\ejklt}{\tilde{e}_{j,k,l}}
\newcommand{\fjklt}{\tilde{f}_{j,k,l}}

\newcommand{\HbzOD}{{H^\beta_0(\OD)}}

\newcommand{\lk}{\lambda_k}
\newcommand{\lkc}{\overline{\lambda}_k}

\newcommand{\lknz}{{\underset{\lk \neq 0}{k=1}}}

\newcommand{\DP}{{\text{DP}}}

\newcommand{\dn}{{\delta_n}}

\newcommand{\Af}{\boldsymbol{A}}
\newcommand{\xf}{\boldsymbol{x}}
\newcommand{\yf}{\boldsymbol{y}}

\newcommand{\Sf}{\boldsymbol{S}}
\newcommand{\ekf}{\boldsymbol{e}_k}
\newcommand{\ektf}{\tilde{\boldsymbol{e}}_k}

\title{On Regularization via Frame Decompositions with Applications in Tomography}
\author{
Simon Hubmer\footnote{Johann Radon Institute Linz, Altenbergerstra{\ss}e 69, A-4040 Linz, Austria, (simon.hubmer@ricam.oeaw.ac.at), Corresponding author.} ,
Ronny Ramlau\footnote{Johannes Kepler University Linz, Institute of Industrial Mathematics, Altenbergerstra{\ss}e 69, A-4040 Linz, Austria, (ronny.ramlau@jku.at)} \footnote{Johann Radon Institute Linz, Altenbergerstra{\ss}e 69, A-4040 Linz, Austria, (ronny.ramlau@ricam.oeaw.ac.at)} ,
Lukas Weissinger\footnote{Johannes Kepler University Linz, Doctoral Program Computational Mathematics, Altenbergerstra{\ss}e 69, A-4040 Linz, Austria (lukas.weissinger@dk-compmath.jku.at)}
}

\begin{document}

% Include the title
\maketitle

% Abstract
\begin{abstract}

In this paper, we consider linear ill-posed problems in Hilbert spaces and their regularization via frame decompositions, which are generalizations of the singular-value decomposition. In particular, we prove convergence for a general class of continuous regularization methods and derive convergence rates under both a-priori and a-posteriori parameter choice rules. Furthermore, we apply our derived results to a standard tomography problem based on the Radon transform.

\end{abstract}

% % % % % % % % % % % % %
% Start of the sections %
% % % % % % % % % % % % %

% % % % % % % % % % % % % %
% Section - Introduction  %
% % % % % % % % % % % % % %
\section{Introduction}

In this paper, we consider linear inverse problems in the standard form
    \begin{equation}\label{Ax=y}
		A x = y  \,,
	\end{equation}
where $A : \, X \to Y$ is a bounded linear operator between real or complex Hilbert spaces $X$ and $Y$. Additionally, we assume that $A$ is compact, which implies both that solving \eqref{Ax=y} is an ill-posed problem, and that there exists a \emph{singular system} $(\sigma_k, v_k, u_k)_{k=1}^\infty$ such that $A$ admits a \emph{singular-value decomposition} (SVD) of the following form (cf.~e.g.,\cite{Engl_Hanke_Neubauer_1996}):
	\begin{equation}\label{SVD_A}
		A x  = \sum_{k=1}^{\infty} \sigma_k \spr{x,v_k}_X u_k \,.
	\end{equation}
Hereby, the \emph{singular values} $\sigma_k$ and the \emph{singular functions} $u_k$, $v_k$ are defined as follows: 
    \begin{enumerate}
        \item The sequence $\Kl{\sigma_k^2}_{k=1}^\infty$ consists of the non-zero eigenvalues of $A^*A$ written down in decreasing order, taking multiplicity into account and observing $\sigma_k > 0$.
        \item The sequence $\Kl{v_k}_{k=1}^\infty$ is a corresponding complete orthonormal system of eigenfunctions, i.e., it satisfies 
        \begin{equation}\label{svd_evprop}
            A^*A v_k = \sigma_k^2 v_k.
        \end{equation} 
        Consequently, it spans $\overline{R(A^*)} = N(A)^\perp$.
        \item The singular functions $u_k$ are defined by $u_k := (1/\sigma_k) A v_k$. Hence, they satisfy $AA^* u_k = \sigma_k^2 u_k$ and form a complete orthonormal system spanning $\overline{R(A)}$.
    \end{enumerate}
Note that due to the above definition, the singular functions $u_k$ and $v_k$ satisfy
    \begin{equation}\label{SVD_uk_vk}
        \sigma_k u_k =  A v_k \,,
        \qquad
        \text{and}
        \qquad
        \sigma_k v_k = A^* u_k \,.
    \end{equation}
The SVD is an important tool for analysing and solving ill-posed problems. In particular, the \emph{minimum-norm least-squares solution} $\xD$ of \eqref{Ax=y} is characterized by
    \begin{equation}\label{def_Ad}
		\xD := A^\dagger y:= \sum_{k=1}^{\infty} \frac{\spr{y,u_k}_Y}{\sigma_k} v_k \,,
	\end{equation}
which is well-defined if and only if the so-called \emph{Picard condition} holds:
	\begin{equation}\label{Picard}
		\sum_{k=1}^{\infty} \frac{\abs{\spr{y,u_k}_Y}^2}{\sigma_k^2} < \infty \,.
	\end{equation}
Furthermore, given noisy data $\yd$, which are typically assumed to satisfy
    	\begin{equation}\label{cond_noise}
		\norm{y - \yd}_Y \leq \delta \,,
	\end{equation}
where $\delta$ denotes the \emph{noise level}, one can define stable approximations $\xad$ of $\xD$ by
	\begin{equation}\label{def_xad}
		\xad := \sum_{k=1}^{\infty} \sigma_k g_\alpha(\sigma_k^2) \spr{\yd,u_k}v_k \,,
	\end{equation}
where $g_\alpha$ is a properly selected approximation of $s \mapsto 1/s$. If the \emph{regularization parameter} $\alpha$ is suitably chosen, e.g.,\ by an a-priori or an a-posteriori parameter choice rule \cite{Engl_1997}, then one can prove that $\xad \to \xD$ as $\delta \to 0$. Furthermore, if, e.g.,\ the \emph{source condition} 
    \begin{equation}\label{sourcecond}
        \xD \in R((A^*A)^\mu) 
    \end{equation}
holds, then one can even prove (order-optimal) convergence rates of the form \cite{Engl_Hanke_Neubauer_1996,Louis_1989}
    \begin{equation*}
        \norm{\xad - \xD}_X = \LandauO\kl{\delta^\frac{2\mu}{2\mu+1}} \,.
    \end{equation*}
Most classic regularization methods can be identified with specific choices of $g_\alpha$, which are also called \emph{spectral filter functions}. For example, for \emph{Tikhonov regularization}, \emph{Landweber iteration}, or the \emph{truncated singular-value decomposition (TSVD)}, there respectively holds \cite{Engl_Hanke_Neubauer_1996,Louis_1989}
    \begin{equation}\label{filters}
    g_\alpha(s) = 
    \begin{cases}
        (s+\alpha)^{-1}\ \,, &\text{Tikhonov,}\\
        s^{-1}\kl{1-(1-s)^{1/\alpha}}\,,&\text{Landweber,}\\
        \begin{cases}
        1/s \,, & s \geq \alpha \,,
        \\
        0 \,, & \text{else} \,,
    \end{cases}
    &\text{TSVD.}
    \end{cases}
    \end{equation}
Unfortunately, for many operators no explicit representation of the SVD is known, and even if it is, its numerical computation might be infeasible (see, e.g.,~\cite{Ramlau_Koutschan_Hofmann_2020}). Additionally, generalizing an available SVD of an operator defined between Hilbert spaces over regular domains to irregular domains is often impossible, e.g.\ in case orthogonality is lost. The situation is somewhat better for finite-dimensional problems, where one can work with the analogously defined SVD of a matrix, but this can also become infeasible when the problem is medium- to large-scale. Hence, even though an important tool in the analysis of ill-posed problems, the SVD often only enjoys limited use in practice.

In order to remedy this situation, several researchers have studied generalizations of the SVD such as the \emph{Wavelet-Vaguelette Decomposition} (WVD) \cite{Donoho_1995,Abramovich_Silverman_1998,Kudryavtsev_Shestakov_2019, Lee_1997,Dicken_Maass_1996, Frikel_Haltmeier_2018} or the \emph{Frame Decomposition} (FD) \cite{Hubmer_Ramlau_2021_01,Hubmer_Ramlau_2020,Ebner_Frikel_Lorenz_Schwab_Haltmeier_2020,Frikel_Haltmeier_2020}. The idea is that by weakening some of the requirements of the SVD such as the eigenvalue properties and the resulting orthogonality of the functions $u_k$ and $v_k$, one may end up with decompositions similar to \eqref{SVD_A} and \eqref{def_Ad} which are easier to derive explicitly for a given operator. For example, the WVD requires an orthonormal wavelet basis $\Kl{\psi_{jk}}_{j,k\in\N}$ on $X$ and two bi-orthogonal sets (vaguelettes) $\Kl{u_{jk}}_{j,k\in\N}$, $\Kl{v_{jk}}_{j,k\in\N}$ on $Y$, connected by the quasi-singular relations 
    \begin{equation}\label{WVD_cond}
        \lambda_{jk} v_{jk} = A \psi_{jk} \,,
        \qquad
        \text{and}
        \qquad
        \lambda_{jk} \psi_{jk} = A^* u_{jk} \,, 
        \qquad
        \forall \, j,k \in \N \,.
    \end{equation}
which are reminiscent of \eqref{SVD_uk_vk}. The resulting decomposition of the operator $A$ is similar to \eqref{SVD_A} and can be used for solving \eqref{Ax=y}. For applications of the WVD to the practically relevant problems of computerized and photoacoustic tomography including different aspects of its numerical realization see e.g.~\cite{Donoho_1995, Frikel_Haltmeier_2018,Frikel_2013}. Extensions of the WVD which also have applications to the two- and three- dimensional Radon transform are e.g.\ given by the biorthogonal curvelet and shearlet decompositions \cite{Candes_Donoho_2002,Colonna_Easley_Guo_Labate_2010}.

In contrast to the WVD, the FD does not specifically work with (orthogonal) wavelets and vaguelettes but with general frames in Hilbert spaces; cf.~Section~\ref{sect_frames}. In particular, the FD requires a frame $\Kl{e_k}_{k\in\N}$ over $X$ and a frame $\Kl{f_k}_{k\in\N}$ over $Y$, connected via
    \begin{equation}\label{cond_frames_connected}
        \lkc \, e_k =  \, A^*f_k \,,
        \qquad
        \forall \,k \in \N \,,
    \end{equation}
where $\lkc$ denotes the complex conjugate of the coefficient $\lk \in \C$. The above condition is clearly connected to \eqref{SVD_uk_vk} and \eqref{WVD_cond}, and leads to the following decomposition:
    \begin{equation}
		A x = \sum_{k=1}^\infty \lk \spr{x,e_k}_X \fkt  \,,
		\qquad \forall \, x \in X \,.
	\end{equation}
Here the functions $\{\fkt\}_{k\in\N}$ denote the dual frame of the frame $\Kl{f_k}_{k\in\N}$; cf.~Section~\ref{sect_frames}.	Furthermore, for $y \in Y$ one can then analogously to \eqref{def_Ad} consider the operator
	\begin{equation}\label{def_AD_expl}
		\AD y  :=   
		 \sum_\lknz^\infty 
		 \frac{1}{\lk} \spr{y,f_k}_Y  \ekt \,,
	\end{equation}
where similarly to above $\{\ekt\}_{k\in\N}$ denotes the dual frame of the frame $\Kl{e_k}_{k\in\N}$. The element $\AD y$ is well-defined if the following analog of the Picard condition \eqref{Picard} holds: 
	\begin{equation}\label{cond_Picard}
		\sum\limits_{\underset{\lk\neq 0}{k = 1}}^\infty\frac{\abs{ \spr{y,f_k}_Y }^2 }{\abs{\lk}^2}
	    < \infty \,.
	\end{equation}
The properties of the operator $\AD$ and its use for obtaining (approximate) solutions of \eqref{Ax=y} was studied in detail in \cite{Hubmer_Ramlau_2021_01}. In particular, it was investigated in which cases $\AD y$ is either a minimum-coefficient or a minimum-norm (least-squares) solution of \eqref{Ax=y}. While we refer to \cite{Hubmer_Ramlau_2021_01} for details, we here only want to mention the special case that $A$ satisfies a  stability condition of the form 
    \begin{equation}\label{cond_A_stability}
	    c_1 \norm{x}_X \leq \norm{Ax}_Z \leq c_2 \norm{x}_X \,,
	    \qquad
	    \forall \, x \in X \,,
	\end{equation}	
for some constants $c_1,c_2 > 0$ with $Z \subseteq Y$ being a Hilbert space. In this case, for each $y \in R(A)$ the unique solution of \eqref{Ax=y} is precisely given by $\AD y$. Furthermore, in this case it is possible to give recipes for finding frames $\Kl{e_k}_{k\in\N}$ and $\Kl{f_k}_{k\in\N}$ which satisfy \eqref{cond_frames_connected}. For example, starting with a frame $\Kl{f_k}_{k\in\N}$ which satisfies the condition
    \begin{equation}\label{cond_norm_Z}
    	a_1 \norm{y}_Z^2 
    	\leq \sum_{k=1}^\infty \ak^2  \abs{\spr{y,f_k}_Y}^2 
    	\leq a_2 \norm{y}_Z^2
    	\,,
    	\qquad
    	\forall \, y \in Y \,,
	\end{equation}
for a sequence of coefficients $0 \neq \ak \in \R$ and some constants $a_1,a_2 > 0$, one can define
    \begin{equation*}
        e_k := \ak A^* f_k \,, 
        \qquad
        \forall \, k \in \N \,,
    \end{equation*}
and it then follows that $\Kl{e_k}_{k\in\N}$ forms a frame over $X$ satisfying \eqref{cond_frames_connected} with $\overline{\lk} = 1/\ak$ \cite{Hubmer_Ramlau_2021_01}. A typical example for an operator satisfying condition \eqref{cond_A_stability} is given by the Radon transform \cite{Louis_1989,Natterer_2001}, and \eqref{cond_norm_Z} can e.g.\ be satisfied if $Y$ and $Z$ are (suitably connected) Sobolev spaces and $\Kl{f_k}_{k\in\N}$ is either an exponential or a wavelet frame/basis. Note also that condition \eqref{cond_A_stability} is satisfied for any continuously invertible operator $A$ with $Z=X$, in which case also \eqref{cond_norm_Z} holds for any frame $\Kl{f_k}_{k\in\N}$ over $Y$. Further details on all of these topics can be found in \cite{Hubmer_Ramlau_2021_01}, which also includes a generalization of condition \eqref{cond_frames_connected} that can be useful if the Hilbert spaces $X$ and $Y$ have a particular product structure \cite{Hubmer_Ramlau_2020, Hubmer_Ramlau_2021_01,Weissinger_2021}. 

In this paper, we focus on different aspects of regularization via FDs. In particular, analogously to \eqref{def_xad} for the SVD we consider stable approximations of $\AD y$ of the form
    \begin{equation*}
        \zad := \sum_\lknz^\infty \lk g_\alpha(\lk^2) \spr{\yd,f_k}_Y \ekt \,,
    \end{equation*}
and derive convergence and convergence rate results under a-prior and a-posteriori parameter choice rules similar to those for the SVD \cite{Engl_Hanke_Neubauer_1996}. We want to note that this work is inspired by our recent investigations of general frame decompositions in Hilbert spaces \cite{Hubmer_Ramlau_2021_01} and their specific application to the atmospheric tomography problem \cite{Hubmer_Ramlau_2020,Weissinger_2021}.

However, we emphasize that different aspects of regularization via WVDs and specific FDs have already been considered in the literature before \cite{Frikel_Haltmeier_2018,Frikel_Haltmeier_2020,Ebner_Frikel_Lorenz_Schwab_Haltmeier_2020}. For example, a WVD approach to photoacoustic tomography was developed in \cite{Frikel_Haltmeier_2018}, where regularization with wavelet sparsity constraints is employed via a soft-thresholding approach. This was generalized to sparse regularization for inverse problems using FDs and nonlinear soft-thresholding in \cite{Frikel_Haltmeier_2020}, where also convergence rates were proven under a-priori parameter choice rules. The analysis in \cite{Frikel_Haltmeier_2020} is based on the assumption that $\Kl{f_k}_{k\in\N}$ forms a frame over $\overline{R(A)}$ instead of over $Y$ as in our current setting. The same is also assumed in the preprint \cite{Ebner_Frikel_Lorenz_Schwab_Haltmeier_2020}, in addition to the requirement that $\Kl{e_k}_{k\in\N}$ forms a frame over $N(A)^\perp = \overline{R(A^*)}$ and that $\lk \in (0,\infty)$ for all $k\in\N$. Note that under these assumptions, for all $y \in D(A^\dagger)$ there holds $A^\dagger y = \AD y$. Working within this setting, the authors of \cite{Ebner_Frikel_Lorenz_Schwab_Haltmeier_2020} prove convergence and convergence rates under a-priori parameter choice rules for continuous regularization methods based on FDs similar to our results. As in this paper, the analysis is based on the standard approach to linear inverse problems \cite{Engl_Hanke_Neubauer_1996}.

In contrast to these papers, here we consider the general FD setup introduced above, i.e., we allow arbitrary $\lk \in \C$ and assume that $\Kl{e_k}_{k\in\N}$ and $\Kl{f_k}_{k\in\N}$ form frames over $X$ and $Y$, respectively. This setting is beneficial, since it allows more general and potentially different frame decompositions of a given operator. Furthermore, it is often easier to find a frame over the whole spaces $X$ and $Y$ instead of over $N(A)^\perp$ and $\overline{R(A)}$, respectively. Of course, one may theoretically obtain such frames by projecting the frames $\{ e_k \}_{k\in\N}$ and $\{ f_k \}_{k\in\N}$ onto these subspaces. However, doing so is practically infeasible, since in most cases an explicit characterization of $N(A)^\perp$ and $\overline{R(A)}$, and thus of the orthogonal projectors onto these subspaces is unavailable or involves the operator $A^\dagger$. This situation is comparable to the SVD of a compact operator: the existence is guaranteed but explicit representations are often unavailable, which is a motivation for considering frames in the first place. Furthermore, while \eqref{cond_frames_connected} implies that $\{ f_k \}_{k\in\N, \lambda_k \neq 0}$ forms a frame over $\overline{R(A)}$, the set $\{ e_k \}_{k\in\N, \lambda_k \neq 0}$ does not necessarily form a frame over $N(A)^\perp$, unless further assumptions on the values $\lambda_k$, the frame $\{f_k\}_{k\in\N}$, and/or the mapping properties of the operator $A$ are made. Note that by redefining the frame functions $e_k$ one could always achieve that $\lambda_k \in \R_0^+$, at the cost of potentially changing the solution properties of the element $\AD y$; compare with the conditions in \cite{Hubmer_Ramlau_2021_01}. However, the fact that $\lk$ may also be zero is a crucial difference to the setting of \cite{Ebner_Frikel_Lorenz_Schwab_Haltmeier_2020}, since then \eqref{cond_frames_connected} no longer implies that all frame functions $e_k$ are elements in $N(A)^\perp = \overline{R(A^*)}$. Hence, every dual frame function $\ekt$ may also have some component outside of $N(A)^\perp$, since by definition it depends on all frame functions $e_k$ including those which correspond to $\lambda_k = 0$. This may translate to the element $\AD y$, except in those situations characterized in \cite{Hubmer_Ramlau_2021_01} in which $\AD y = A^\dagger y \in N(A)^\perp$. The benefit of allowing $\lambda_k = 0$ is that it can make the search for frames satisfying \eqref{cond_frames_connected} easier, as is the case e.g.\ for the atmospheric tomography problem \cite{Hubmer_Ramlau_2020}. In contrast, under the more restrictive assumptions of \cite{Ebner_Frikel_Lorenz_Schwab_Haltmeier_2020} there always holds $A^\dagger y = \mathcal{A}y$, which is not necessarily the case in our more general setting (see above). Hence, while \cite{Ebner_Frikel_Lorenz_Schwab_Haltmeier_2020} considers the stable approximation of $A^\dagger y$ in the presence of noisy data via its representation in terms of frames, here we consider the stable approximation of the (approximate) solution $\mathcal{A} y$. Finally, note that in contrast to \cite{Ebner_Frikel_Lorenz_Schwab_Haltmeier_2020} we also present convergence rates results for an a-posteriori parameter choice rule adapted from the discrepancy principle, as well as numerical results illustrating our derived theory on the example of a standard tomography problem based on the Radon transform. 

The outline of this paper is as follows: In Section~\ref{sect_Regularizations}, after reviewing some necessary material on frames in Hilbert spaces, we consider continuous regularization methods based on frame decompositions and show convergence and convergence rates under standard assumptions, both under a-priori and a-posteriori parameter choice rules. In Section~\ref{sect_Radon_Trans}, we then apply our results to a standard tomography problem based on the Radon transform, providing numerical examples for specific FDs and comparing the results of different regularization methods. Section~\ref{sect_conclusion} then summarizes our results.

% % % % % % % % % % % % % % % % % % % % % % % % % % %
% Section - Regularization via Frame Decompositions %
% % % % % % % % % % % % % % % % % % % % % % % % % % %
\section{Regularization via Frame Decompositions}\label{sect_Regularizations}

% Subsection - Background Frames in Hilbert Spaces
\subsection{Background on Frames in Hilbert Spaces}\label{sect_frames}

Before deriving our results on regularization via FDs, we first recall some basic facts on frames in Hilbert spaces. This short summary, based on the seminal work \cite{Daubechies_1992}, is adapted from our previous publications \cite{Hubmer_Ramlau_2020,Hubmer_Ramlau_2021_01}. First, recall the definition of a frame.

\begin{definition}\label{def_frame}
A sequence $\{e_k\}_{k \in \N}$ in a Hilbert space $X$ is called a frame over $X$, if and only if there exist \emph{frame bounds} $0< B_1,B_2 \in \R$ such that for all $x \in X$ there holds
	\begin{equation}\label{eq_framedef}
    	B_1 \norm{x}_X^2 \leq \sum\limits_{k=1}^\infty \abs{\spr{x,e_k}_X}^2 \leq B_2 \norm{x}_X^2 \,.
	\end{equation}
\end{definition}

For a given frame $\{e_k\}_{k\in \N}$ one can consider the \emph{frame (analysis) operator} $F$ and its adjoint \emph{(synthesis)} operator $F^*$, which are given by
    \begin{equation}\label{def_F_Fad}
    \begin{split}
        &F \, : \, X \to \lt(\N) \,, \qquad
        x \mapsto \Kl{\spr{x,e_k}_X}_{k \in \N} \,,
        \\ 
        &F^* \, : \, \lt(\N) \to X \,, \qquad
        \Kl{a_k}_{k\in\N} \mapsto \sum\limits_{k=1}^\infty a_k e_k  \,.
    \end{split}	
    \end{equation}
Due to \eqref{eq_framedef} there holds
    \begin{equation}\label{eq_bound_F_Fadj}
        \sqrt{B_1} \leq \norm{F} = \norm{F^*} \leq \sqrt{B_2} \,.
    \end{equation}
Furthermore, one can define the operator $S := F^*F$, i.e.,
    \begin{equation}\label{Def_S}
        S x := \sum\limits_{k=1}^\infty \spr{x,e_k}_X e_k \,,
    \end{equation}
which is a bounded and continuously invertible linear operator with $B_1 I \leq S \leq B_2 I$ and $B_2^{-1} I \leq S^{-1} \leq B_1^{-1} I$. Hence, it follows that with $\et_k := S^{-1}e_k$ there holds
    \begin{equation}\label{eq_frame}
    	B_2^{-1} \norm{x}_X^2 \leq \sum\limits_{k=1}^\infty \abs{\spr{x,\ekt}_X}^2 \leq B_1^{-1} \norm{x}_X^2 \,,
	\end{equation}
for all $x \in X$, and thus the set $\{\et_k\}_{k\in \N}$ also forms a frame over $X$ with frame bounds $B_2^{-1},B_1^{-1}$ which is called the \emph{dual frame} of $\{e_k\}_{k \in \N}$. For the corresponding operators
    \begin{equation}\label{def_F_Fad_t}
    \begin{split}
        & \Ft \, : \, X \to \lt(\N) \,, \qquad
        x \mapsto \Kl{\spr{x,\ekt}_X}_{k \in \N} \,,
        \\ 
        &\Ft^* \, : \, \lt(\N) \to X \,, \qquad
        \Kl{a_k}_{k\in\N} \mapsto \sum\limits_{k=1}^\infty a_k \ekt  \,.
    \end{split}
    \end{equation}
it follows analogously to \eqref{eq_bound_F_Fadj} that 
    \begin{equation}\label{eq_bound_F_Fadj_t}
        \sqrt{1/B_2} \leq \norm{\Ft} = \norm{\Ft^*} \leq \sqrt{1/B_1} \,.
    \end{equation}
In particular, for any sequence of coefficients $a=\{a_k\}_{k\in\N}$ there holds
    \begin{equation}\label{Fadj_norm}
        \norm{ \sum_{k\in\N}a_k\tilde{e}_k }_X
        =
        \norm{\tilde{F}^\ast a}_X
        \leq 
        \sqrt{1/B_1}\norm{a}_{\lt(\N)}.
    \end{equation}
Furthermore, it can be shown that there holds
    \begin{equation*}
        \Ft^*F = F^* \Ft = I \,,
    \end{equation*}
and thus any $x \in X$ can be written in the form
    \begin{equation}\label{eq_frame_rec}
        x = \sum\limits_{k=1}^\infty \spr{x, \et_k}_X e_k
        = \sum\limits_{k=1}^\infty \spr{x, e_k}_X \et_k \,.
    \end{equation}
Note that for any frame $\Kl{e_k}_{k\in\N}$ the following statements are equivalent (see, e.g.,~\cite{Christensen_2016}):
    \begin{equation}\label{eq_frame_equivalences}
    \begin{split}
        N(F^*) = 0 
        \quad &\Leftrightarrow \quad
        \Kl{e_k}_{k\in\N} \text{ is a (Riesz) basis}
        \\
        \quad &\Leftrightarrow \quad
        \Kl{e_k}_{k\in\N} \text{ and } \Kl{\ekt}_{k\in\N} \text{ are biorthogonal}
        \\
        \quad &\Leftrightarrow \quad
        \Kl{e_k}_{k\in\N} \text{ is exact (i.e.\ no element can be deleted)} \,.
    \end{split}
    \end{equation}    
In general there holds holds $\Kl{0} \subset \,  N(F^*) = N(\Ft^*)$, an thus the decomposition of $x$ given in \eqref{eq_frame_rec} is not unique, which is a key differences between frames and bases. However, this decomposition can be understood as the \emph{most economical} one (cf.~\cite{Daubechies_1992}).

% % % % % % % % % % % % % % % % % % % % % % % % % %
% Subsection - Continuous Regularization Methods  %
% % % % % % % % % % % % % % % % % % % % % % % % % %
\subsection{Continuous Regularization Methods}

In this section, we consider general continuous regularization methods based on FDs similar to those based on the SVD \cite{Engl_Hanke_Neubauer_1996}. More precisely, for $\alpha > 0$ we consider the functions
	\begin{equation}\label{def_zad}
		\zad = \sum_\lknz^\infty \lk g_\alpha(\lk^2) \spr{\yd,f_k}_Y \ekt \,,
	\end{equation}
as approximations of $\AD y$ given noisy data $\yd$, as well as the functions
	\begin{equation}\label{def_za}
		\za = \sum_\lknz^\infty \lk g_\alpha(\lk^2) \spr{y,f_k}_Y \ekt \,,
	\end{equation}
in the noise-free case. Here, $g_\alpha : \C \to \C$ is a suitable approximation of $s\mapsto 1/s$ to be specified below. We will prove convergence as well as convergence rates of the form
    \begin{equation*}
        \norm{\AD y - \zad}_Y = \LandauO\kl{\delta^{\frac{2\mu}{2\mu+1}}}
    \end{equation*}   
for both a-priori and a-posterior parameter choice rules under standard assumptions. Throughout the analysis, which is based on classical arguments (see e.g.~\cite{Engl_Hanke_Neubauer_1996}), we use

\begin{assumption}\label{assumption_main}
The operator $A: X \to Y$ is bounded, linear, and compact between the (complex) Hilbert spaces $X$ and $Y$. Furthermore, the set $\{e_k\}_{k\in\N}$ forms a frame over $X$ with frame bounds $B_1, B_2$, and the set $\{f_k\}_{k\in\N}$ forms a frame over $Y$ with frame bounds $C_1, C_2$. Moreover, there exist coefficients $\lk \in \C$ such that \eqref{cond_frames_connected} holds. In addition, let $g_\alpha : \C \to \C$ be a parameter-dependent family of piecewise continuous, bounded functions defined for all $\alpha > 0$, satisfying
    \begin{equation}\label{ass_ga_limit}
        \lim_{\alpha \to 0} g_\alpha(\lambda) = \frac{1}{\lambda} \,,
        \qquad
        \forall \, \lambda \in \C \,.
    \end{equation}
Additionally, assume there exists a constant $C>0$ independent of $\alpha$ such that
    \begin{equation}\label{ass_ga_bound}
        \abs{\lambda g_\alpha(\lambda)} \leq C \,,
        \qquad
        \forall \, \lambda \in \C \,.
    \end{equation}
\end{assumption}
Note that \eqref{ass_ga_limit} and \eqref{ass_ga_bound}, as well as the following definitions, which we need for the upcoming analysis, are the same as those used in the classic SVD analysis \cite{Engl_Hanke_Neubauer_1996,Louis_1989}:

\begin{definition}
For all $\alpha > 0$ we define
    \begin{equation}\label{def_ra}
       r_\alpha(\lambda) := 1- \lambda g_\alpha(\lambda) \,, 
    \end{equation}
as well as
    \begin{equation}\label{def_Ga}
        G_\alpha := \sup \Kl{ \abs{g_\alpha(\lambda)} \, \vert \, \lambda \in \C} \,.
    \end{equation}
Furthermore,
    \begin{equation}\label{def_gamma}
        \gamma := \sup\Kl{ \abs{r_\alpha(\lambda) } \, \vert\, \alpha > 0\,, \lambda \in \C }
        \overset{\eqref{ass_ga_bound}}{\leq} (C+1) \,.
    \end{equation}
\end{definition}

Next, concerning the well-definedness of $\zad$ and $\za$ we have
\begin{lemma}\label{lem_welldef}
Let Assumption~\ref{assumption_main} hold and let $y,\yd \in Y$. Then $\zad$ and $\za$ as given in \eqref{def_zad} and \eqref{def_za} are well-defined. Furthermore, if additionally the Picard condition \eqref{cond_Picard} holds, then $\AD y$ given in \eqref{def_AD_expl} is well-defined.
\end{lemma}
\begin{proof}
Let $y \in Y$ be arbitrary but fixed. Since the set $\{e_k\}_{k\in\N}$ forms a frame over $X$ with frame bounds $B_1,B_2$, it follows from \eqref{Fadj_norm} that
    \begin{equation*}
        \norm{\za}_Y^2 
        \leq 
        (1/B_1) \sum_\lknz^\infty \abs{ \lk g_\alpha(\lk^2) \spr{y,f_k}_Y }^2
        \overset{\eqref{ass_ga_bound}}{\leq}
        (C/B_1) \norm{g_\alpha}_\infty \sum_\lknz^\infty \abs{\spr{y,f_k}_Y }^2
        \,.
    \end{equation*}
Furthermore, since $\{f_k \}_{k\in\N}$ forms a frame over $Y$ with bounds $C_1,C_2$, it follows that
    \begin{equation*}
        \norm{\za}_Y^2 
        \leq C_2 (C /B_1) \norm{g_\alpha}_\infty \norm{y}_Y^2 \,.
    \end{equation*}
Hence, since we assumed that $g_\alpha$ is bounded it follows that $\za$ is well-defined. Analogously we can show that $\zad$ is well-defined. Furthermore, similarly to above we have
    \begin{equation*}
        \norm{\AD y}_Y^2 \leq (1/B_1) \sum_\lknz^\infty \abs{ \frac{\spr{y,f_k}_Y}{\lk} }^2 
        \overset{\eqref{cond_Picard}}{<} \infty \,.
    \end{equation*}
Hence, if \eqref{cond_Picard} is satisfied then also $\AD y$ is well-defined, which concludes the proof.
\end{proof}

The following result establishes convergence of $\za$ to $\AD y$ in the noise-free case.
\begin{theorem}\label{thm_conv_noisefree}
Let Assumption~\ref{assumption_main} hold and let $y \in Y$ satisfy \eqref{cond_Picard}. Then for $\AD y$ and $\za$ as defined in \eqref{def_AD_expl} and \eqref{def_za}, respectively, there holds
	\begin{equation*}
		\lim_{\alpha \to 0} \norm{\AD y - \za}_X = 0 \,.
	\end{equation*}
\end{theorem}
\begin{proof}
Let $y\in Y$ be arbitrary but fixed and let \eqref{cond_Picard} hold. Then due to Lemma~\ref{lem_welldef} both $\AD y$ and $\za$ are well-defined. Furthermore, due to \eqref{def_AD_expl} and \eqref{def_za} there holds 
    \begin{equation}\label{helper_eq_1}
		\AD y - \za = \sum_\lknz^\infty (1/\lk - \lk g_\alpha(\lk^2) )\spr{y,f_k}_Y \ekt
		=
		\sum_\lknz^\infty \frac{r_\alpha(\lk^2) }{\lk} \spr{y,f_k}_Y \ekt \,.
	\end{equation}
Since $\{e_k\}_{k\in\N}$ forms a frame with frame bounds $B_1,B_2$, it follows from \eqref{Fadj_norm} that
    \begin{equation}\label{helper_ineq_1}
		\norm{\AD y - \za}_Y^2 \leq \frac{1}{B_1} \sum_\lknz^\infty \abs{\frac{r_\alpha(\lk^2) }{\lk} \spr{y,f_k}_Y }^2 \,.
	\end{equation}
Since due to \eqref{ass_ga_bound} and the definition of $r_\alpha$ there holds $\abs{r_\alpha(\lambda)} \leq (C + 1)$, it follows that
    \begin{equation*}
        \norm{\AD y - \za}_Y^2 \leq \frac{(C+1)^2}{B_1}  \sum_\lknz^\infty \frac{\abs{\spr{y,f_k}_Y}^2 }{\abs{\lk}^2} \,.
	\end{equation*}
The right-hand side in the above inequality is uniformly bounded independently of $\alpha$ due to \eqref{cond_Picard}. Hence, we can apply the dominated convergence theorem to obtain
    \begin{equation*}
		\lim_{\alpha \to 0} \norm{\AD y - \za}_X^2 
		\overset{\eqref{helper_ineq_1}}{\leq}
		\lim_{\alpha \to 0} \frac{1}{B_1} \sum_\lknz^\infty \abs{\frac{r_\alpha(\lk^2) }{\lk} \spr{y,f_k}_Y }^2 
		\leq
		\frac{1}{B_1} \sum_\lknz^\infty \lim_{\alpha \to 0} \abs{r_\alpha(\lk^2)}^2 \frac{\abs{\spr{y,f_k}_Y}^2}{\abs{\lk}^2} \,.
	\end{equation*}
Since due to \eqref{ass_ga_limit} there holds $\lim_{\alpha \to 0} r_\alpha(\lambda) = 0 $ for all $\lambda \neq 0 $, it follows that
	\begin{equation*}
		\lim_{\alpha \to 0}\norm{\AD y - \za }_X = 0 \,,
	\end{equation*}
which yields the assertion and thus concludes the proof.
\end{proof}
Next, we derive an upper bound on the data-propagation error in the following

\begin{theorem}\label{thm_stability_noise}
Let Assumption~\ref{assumption_main} hold, let $y,\yd \in Y$ satisfy \eqref{cond_noise}, and let $\za$, $\zad$ be as in \eqref{def_za}, \eqref{def_zad}, respectively. Then with $G_\alpha$ as defined in \eqref{def_Ga} there holds
	\begin{equation*}
		\norm{\za - \zad}_X \leq \delta \sqrt{C C_2/ B_1 } \sqrt{G_\alpha} \,.  
	\end{equation*}
\end{theorem}
\begin{proof}
Let $y,\yd \in Y$ satisfying \eqref{cond_noise} be arbitrary but fixed. Then due to Lemma~\eqref{lem_welldef} both $\za$ and $\zad$ are well-defined. By their definitions \eqref{def_za} and \eqref{def_zad} it follows that
    \begin{equation*}
        \za - \zad = \sum_\lknz^\infty \lk g_\alpha(\lk^2) \spr{y- \yd,f_k}_Y \ekt \,.
    \end{equation*}
Since $\{e_k\}_{k\in\N}$ forms a frame with frame bounds $B_1,B_2$, it follows from \eqref{Fadj_norm} that
	\begin{equation*}
		\norm{\za - \zad}_X^2 
		\leq
		\frac{1}{B_1}  
		\sum_\lknz^\infty \abs{\lk g_\alpha(\lk^2)}^2 \abs{\spr{y - \yd,f_k}_Y}^2  \,,
	\end{equation*}
and thus together with the definition \eqref{def_Ga} of $G_\alpha$ and \eqref{ass_ga_bound} we get		
\begin{equation*}
		\norm{\za - \zad}_X^2 
		\leq
		\frac{C G_\alpha }{B_1}
		\sum_\lknz^\infty \abs{\spr{y - \yd,f_k}_Y}^2  \,.
	\end{equation*}
Hence, since $\{f_k\}_{k\in\N}$ forms a frame with frame bounds $C_1,C_2$, it follows that
    \begin{equation*}
        \norm{\za - \zad}_X^2 
		\leq
		C G_\alpha (C_2/B_1)
		\norm{y - \yd}_Y^2
		\overset{\eqref{cond_noise}}{\leq}
		C G_\alpha (C_2/B_1)
		\delta^2 \,,
    \end{equation*}	
which after taking the square root on both sides yields the assertion.
\end{proof}

Combining Theorem~\ref{thm_conv_noisefree} and Theorem~\ref{thm_stability_noise}, for the total error we obtain
	\begin{equation}\label{helper_ineq_2}
		\norm{\AD y - \zad}_X \leq \norm{\AD y - \za}_X + \delta \sqrt{C C_2 /B_1} \sqrt{G_\alpha}  \,, 
	\end{equation}
where the first term tends to $0$ for $\alpha \to 0$ if $y$ satisfies \eqref{cond_Picard}, similarly as in the SVD case. Hence, if the regularization parameter $\alpha = \alpha(\delta)$ is chosen such that
    \begin{equation*}
        \alpha(\delta) \to 0 \,, 
        \qquad
        \text{and}
        \qquad
        \delta \sqrt{G_{\alpha(\delta)}} \to 0 \,,
        \qquad
        \text{as}
        \quad 
        \delta \to 0 \,,
    \end{equation*}
then we obtain convergence of $z_{\alpha(\delta)}^\delta$ to $\AD y$ as the noise level $\delta \to 0$.

% % % % % % % % % % % % % % % % % % % % % % % % %
% Subsection - A-priori Parameter Choice Rules  %
% % % % % % % % % % % % % % % % % % % % % % % % %
\subsection{A-priori Parameter Choice Rules}

In this section we derive convergence rates results under a-priori parameter choice rules similar to those for the SVD case (cf.~\cite{Engl_Hanke_Neubauer_1996}). These typically require source conditions such as the H\"older-type condition \eqref{sourcecond} for some $\mu > 0$. Alternatively, this can be rewritten as 
    \begin{equation*}
		\exists \, w \in X : 
		\quad
		\xD = (A^*A)^\mu w \,.
	\end{equation*}
Using the SVD \eqref{svd_evprop} of the operator $A$ it follows
    \begin{equation*}
        \sigma_k^{-1} \spr{y,u_k}_Y
        = 
        \spr{\xD,v_k}_X 
        = 
        \spr{(A^*A)^\mu w,v_k}_X 
		= 
		\sigma_k^{2\mu} \spr{w,v_k}_X \,,
    \end{equation*}
and thus \eqref{sourcecond} is equivalent to
    \begin{equation}\label{sourcecond_SVD_w}
        \exists \, w \in X \,\, \forall \, k \in \N : 
		\quad
		\spr{y,u_k}_Y = \spr{(A^*A)^\mu w,v_k}_X 
		= \sigma_k^{2\mu+1} \spr{w,v_k} \,,
    \end{equation}
which in turn is equivalent to the decay condition \cite[Prop. 3.13]{Engl_Hanke_Neubauer_1996}
    \begin{equation}\label{decay_SVD}
        \sum_{k=1}^\infty \sigma_k^{-(4\mu+2)} \abs{\spr{y,u_k}_Y}^2 < \infty \,.
    \end{equation}
For the upcoming analysis, we use a source condition similar to \eqref{sourcecond_SVD_w}, namely
    \begin{equation}\label{sourcecond_frames}
		\exists \, w \in X \,\, \forall \, k \in \N \,, \lk \neq 0 : 
		\quad
		\spr{y,f_k}_Y = \lk^{2\mu+1} \spr{w,e_k}_X \,, 
	\end{equation}
which analogously to \eqref{decay_SVD} implies the decay condition
	\begin{equation*}
		\sum_{\lknz}^\infty \abs{\lk}^{-(4\mu+2)} \abs{\spr{y,f_k}_Y}^2 < \infty  \,.
	\end{equation*}
We now start our analysis by deriving a convergence rate estimate given exact data.
\begin{theorem}\label{thm_rates_exact_apriori}
Let Assumption~\ref{assumption_main} hold, let $y \in Y$ satisfy \eqref{cond_Picard}, and let $\AD y$ and $\za$ be as in \eqref{def_AD_expl} and \eqref{def_za}, respectively. Moreover, let $\mu > 0$, $\alpha_0 > 0$, and assume that for all $\alpha \in (0,\alpha_0)$ and $\lambda \in \C$ the function $r_\alpha$ defined in \eqref{def_ra} satisfies
	\begin{equation}\label{cond_ra}
		\abs{\lambda}^{\mu} \abs{r_\alpha(\lambda)} \leq c_{\mu} \alpha^{\mu} \,,
	\end{equation}
for some constant $c_{\mu} > 0$. Then if the source condition \eqref{sourcecond_frames} holds it follows that
	\begin{equation*}
		\norm{\AD y - \za}^2
		\leq
		\Kl{\frac{B_2}{B_1} c_{\mu}^2 \norm{w}_X^2 } \alpha^{2\mu}  \,.
	\end{equation*}
\end{theorem}
\begin{proof}
Let $y \in Y$ be arbitrary but fixed and let \eqref{cond_Picard} hold. Then due to Lemma~\ref{lem_welldef} both $\AD y$ and $\za$ are well-defined. Then due to \eqref{helper_eq_1} there holds
    \begin{equation*}
        \AD y - \za = \sum_\lknz^\infty \frac{r_\alpha(\lk^2) }{\lk} \spr{y,f_k}_Y \ekt \,,
    \end{equation*}
which together with the source condition \eqref{sourcecond_frames} implies
	\begin{equation*}
		\AD y - \za = \sum_\lknz^\infty r_\alpha(\lk^2) \lk^{2\mu} \spr{w,e_k}_X \ekt     \,.
	\end{equation*}
Hence, since $\{e_k\}_{k\in\N}$ forms a frame with bounds $B_1,B_2$, it follows with \eqref{Fadj_norm} that	
	\begin{equation*}
		\norm{\AD y - \za}_X^2 \leq \frac{1}{B_1}  \sum_{\lknz}^\infty 
		\abs{ r_\alpha(\lk^2) \lk^{2\mu}}^2 \abs{\spr{w,e_k}_X}^2 \,.
	\end{equation*}
Together with \eqref{cond_ra} this implies
    \begin{equation*}
		\norm{\AD y - \za}_X^2 
		\overset{\eqref{cond_ra}}{\leq} 
		\frac{1}{B_1}(c_{\mu} \alpha^{\mu})^2 \sum_{\lknz}^\infty 
		\abs{\spr{w,e_k}_X}^2 
		\overset{\eqref{def_frame}}{\leq}
		\frac{B_2}{B_1} (c_{\mu} \alpha^{\mu})^2 \norm{w}_X^2 \,,
	\end{equation*}
which yields the assertion.
\end{proof}

Next, we derive convergence rate estimates also in the case of inexact data.

\begin{theorem}\label{thm_rates_noisy_apriori} Let Assumption~\ref{assumption_main} hold, let $y \in Y$ satisfy \eqref{cond_Picard}, $\yd \in Y$ satisfy \eqref{cond_noise}, and let $\AD y$ and $\zad$ be as in \eqref{def_AD_expl} and \eqref{def_zad}, respectively. Moreover, let $\mu > 0$, $\alpha_0 > 0$, and assume that for all $\alpha \in (0,\alpha_0)$ and $\lambda \in \C$ the function $r_\alpha$ defined in \eqref{def_ra} satisfies \eqref{cond_ra} for some constant $c_{\mu} > 0$. Furthermore, let \eqref{sourcecond_frames} hold and assume that
    \begin{equation}\label{cond_Ga}
        G_\alpha = \LandauO\kl{\alpha^{-1}} \quad \text{as} \quad \alpha \to 0 \,.
    \end{equation}
Then with the a-priori parameter choice rule
    \begin{equation}\label{apriori_choice}
        \alpha \sim \delta^{\frac{2}{2\mu+1}} \,,
    \end{equation}
it follows that
    \begin{equation*}
        \norm{\AD y - \zad}_X = \LandauO\kl{ \delta^{\frac{2\mu}{2\mu + 1}}  } \,.
    \end{equation*}
\end{theorem}
\begin{proof}
Let $y,\yd \in Y$ satisfying \eqref{cond_noise}, \eqref{cond_Picard} be arbitrary but fixed. Then by Lemma~\ref{lem_welldef} it follows that $\za$, $\zad$, and $\AD y$ are well-defined. Furthermore, since there holds
    \begin{equation*}
        \norm{\AD y - \zad}_Y \leq \norm{\AD y - \za}_Y + \norm{\za - \zad}_Y \,,
    \end{equation*}
we can apply Theorem~\ref{thm_stability_noise} and Theorem~\ref{thm_rates_exact_apriori} to obtain
    \begin{equation*}
        \norm{\AD y - \zad}_Y \leq \kl{B_2/B_1}^{1/2} c_{\mu} \norm{w}_X \alpha^{\mu} + \delta \sqrt{C C_2 /B_1} \sqrt{G_\alpha} \,.
    \end{equation*}
Together with \eqref{cond_Ga} this implies that
    \begin{equation*}
        \norm{\AD y - \zad}_Y = \LandauO\kl{ \alpha^{\mu} + \frac{\delta}{\sqrt{\alpha}}  } \,,
    \end{equation*}
which together with the a-priori choice \eqref{apriori_choice} yields
    \begin{equation*}
        \norm{\AD y - \zad}_Y = \LandauO\kl{ \delta^{\frac{2\mu}{2\mu+1}}} \,,
    \end{equation*}
and thus concludes the proof.
\end{proof}

\begin{remark}
Note that \eqref{cond_ra} could be replaced by the more general condition
    \begin{equation*}
		\abs{\lambda}^{\mu} \abs{r_\alpha(\lambda)} \leq \omega_{\mu}(\alpha) \,,
	\end{equation*}
for some function $\omega_{\mu} : (0,\alpha_0) \to \R$. With this, one can derive similar results as in Theorem~\ref{thm_rates_exact_apriori} and Theorem~\ref{thm_rates_noisy_apriori} also under more general source conditions than \eqref{sourcecond_frames}.
\end{remark}

% % % % % % % % % % % % % % % % % % % % % % % % % %
% Subsection - A-posteriori Parameter Choice Rule %
% % % % % % % % % % % % % % % % % % % % % % % % % %
\subsection{A-posteriori Parameter Choice Rule}

In this section we derive convergence rates results under an a-posteriori parameter choice rule similar to the SVD case. More precisely, we consider a variant of the well-known \emph{discrepancy principle}, which defines the regularization parameter $\alpha^\DP(\delta,\yd)$ via
    \begin{equation}\label{discrepancy_SVD}
        \alpha^\DP = \alpha^\DP(\delta,\yd) := \sup\Kl{\alpha > 0 \, \vert \, \norm{A \xad - \yd}_Y \leq \tau^\DP \delta} \,,
    \end{equation}
where the constant $\tau^\DP$ is chosen such that
    \begin{equation*}
        \tau^\DP > \sup\Kl{ \abs{r_\alpha(\lambda) } \, \vert\, \alpha > 0\,, \lambda \in \C }  \,.
    \end{equation*}
Since from the properties of the SVD it follows (after some computation) that
    \begin{equation*}
    \begin{split}
    &\norm{A\xad - \yd}_Y^2 
        =
        \sum_{k=1}^\infty \abs{r_\alpha(\sigma_k^2)\spr{\yd,u_k}_Y}^2 + \norm{(I-Q)\yd}_Y^2 \,,
    \end{split}
    \end{equation*}
where $Q$ is the orthogonal projector onto $\overline{R(A)}$, it follows that \eqref{discrepancy_SVD} is equivalent to
    \begin{equation*}
        \alpha^\DP(\delta,\yd) := \sup\Big\{\alpha > 0 \, \Big\vert \,\sum_{k=1}^\infty \abs{r_\alpha(\sigma_k^2)\spr{\yd,u_k}_Y}^2 + \norm{(I-Q)\yd}_Y^2 \leq (\tau^\DP \delta)^2\Big\} \,.
    \end{equation*}
This motivates the following a-posteriori parameter choice rule for the FD case:

\begin{definition}
Let $\gamma$ be as in \eqref{def_gamma} and let the parameter $\tau > 0$ be such that
    \begin{equation}\label{tau}
        \tau > \sqrt{C_2} \gamma \,,
    \end{equation}
where as before $C_2$ denotes the upper frame bound of $\{ f_k \}_{k\in\N}$. Then we define
    \begin{equation}\label{discrepancy_FD}
        \alpha(\delta,\yd) := \sup\Big\{\alpha > 0 \, \Big\vert \, 
        \sum_\lknz^\infty  \abs{r_\alpha(\lk^2) \spr{\yd,f_k}_Y }^2 \leq (\tau \delta)^2\Big\} \,.
    \end{equation}
In case that $\alpha(\delta,\yd) = + \infty$, $z_{\alpha(\delta,\yd)}^\delta$ is understood in the sense of a limit, i.e., 
    \begin{equation}\label{def_zinfty}
        z_\infty^\delta := \lim_{\alpha \to \infty} z_\alpha^\delta \,.
    \end{equation}
\end{definition}

\begin{remark}
While the SVD is generally not known explicitly and thus mostly used as a theoretical tool, we here assume that the FD is known explicitly. Hence, the sum in \eqref{discrepancy_FD} can be computed and thus our a-posteriori rule can also be used in practice.
\end{remark}

Concerning the well-definedness of our parameter choice rule, we have the following

\begin{lemma}\label{lem_welldef_F}
Let Assumption~\ref{assumption_main} hold, let $y,\yd \in Y$ satisfy \eqref{cond_Picard},\eqref{cond_noise}, respectively, and let the function $\alpha \mapsto g_\alpha(\lambda)$ be continuous from the left for all $\lambda \in \C$. Then the set
    \begin{equation}\label{def_D}
        D := \Big\{\alpha > 0 \, \Big\vert \, 
        \sum_\lknz^\infty  \abs{r_\alpha(\lk^2) \spr{\yd,f_k}_Y }^2 \leq (\tau \delta)^2\Big\}
    \end{equation}
is non-empty, and thus \eqref{discrepancy_FD} yields a well-defined stopping index $\alpha(\delta,\yd)$ in $(0,\infty]$. Furthermore, if $\alpha(\delta,\yd) < \infty$ then the supremum in \eqref{discrepancy_FD} is attained, and thus
    \begin{equation}\label{discr_FD_attained}
        \sum_\lknz^\infty  \abs{r_{\alpha(\delta,\yd)}(\lk^2) \spr{\yd,f_k}_Y }^2 \leq (\tau \delta)^2 \,.
    \end{equation}
Moreover, if additionally $G_\alpha$ as defined in \eqref{def_Ga} satisfies 
    \begin{equation}\label{cond_Ga_bound}
        \exists \, \hat{c} \,\, \forall \, \alpha > 0: \quad G_\alpha \leq \hat{c}/\alpha \,,
    \end{equation}
then it follows that $z_\infty^\delta$ as defined in \eqref{def_zinfty} satisfies $z_\infty^\delta = 0$.
\end{lemma}
\begin{proof}
Note first that due to \eqref{def_frame} and \eqref{def_gamma}, for all $\alpha > 0$ there holds
    \begin{equation*}
        \sum_{\lknz}^\infty \abs{r_\alpha(\lk^2)\spr{\yd,f_k}_Y }^2 
        \leq
        \gamma^2 C_2 \norm{\yd}_Y^2 \,.
    \end{equation*}
Hence, we can apply the dominated convergence theorem to obtain
    \begin{equation*}
        \lim_{\alpha \to 0} \sum_{\lknz}^\infty \abs{r_\alpha(\lk^2)\spr{\yd,f_k}_Y }^2
        =
        \sum_{\lknz}^\infty \lim_{\alpha \to 0} \abs{r_\alpha(\lk^2)}^2 \abs{\spr{\yd,f_k}_Y }^2
        \overset{\eqref{ass_ga_limit}}{\underset{\eqref{def_ra}}{=}}
        0\,.
    \end{equation*}
Hence, for each $\eps > 0$ there exists an $\alpha(\eps)$ such that
    \begin{equation*}
        \sum_{\lknz}^\infty \abs{r_{\alpha(\eps)}(\lk^2)\spr{\yd,f_k}_Y }^2 \leq \eps \,,
    \end{equation*}
and thus $D$ defined in \eqref{def_D} is non-empty and consequently $\alpha(\delta,\yd)$ is a well-defined element in $(0,\infty]$. Next, note that since we assumed that for all $\lambda\in\C$ the functional $\alpha \mapsto g_\alpha(\lambda)$ is continuous from the left, it follows that the same is also true for the functional
    \begin{equation*}
        \alpha \mapsto \sum_\lknz^\infty  \abs{r_\alpha(\lk^2) \spr{\yd,f_k}_Y }^2 \,.
    \end{equation*}
Hence, the supremum in \eqref{discrepancy_FD} is attained and thus \eqref{discr_FD_attained} holds whenever $\alpha(\delta,\yd) < \infty$. For the limit case $\alpha(\delta,\yd) = \infty$ observe that due to \eqref{Fadj_norm} and \eqref{def_Ga} there holds
    \begin{equation*}
         \norm{\zad}_X^2  
         \leq 
         \frac{1}{B_1} \sum_\lknz^\infty \abs{\lk g_\alpha(\lk^2) \spr{\yd,f_k}_Y}^2 
         \leq 
         \frac{C G_\alpha}{B_1}  \sum_\lknz^\infty \abs{\spr{\yd,f_k}_Y}^2 \,. 
    \end{equation*}
Hence, together with \eqref{def_frame} and \eqref{cond_Ga_bound} we obtain
    \begin{equation*}
         \norm{\zad}_X^2  
         \leq 
         \frac{C G_\alpha}{B_1}  \sum_\lknz^\infty \abs{\spr{\yd,f_k}_Y}^2
         \leq
        \frac{C_2 C}{B_1} \frac{\hat{c}}{\alpha} \norm{\yd}_Y^2 
         \,,
    \end{equation*}
and thus taking the limit we obtain $z_\infty^\delta = \lim\limits_{\alpha\to \infty} \zad = 0$, which concludes the proof.
\end{proof}

We are now able to derive convergence rate estimates in the presence of noisy data.

\begin{theorem}\label{thm_rates_noisy_aposteriori} Let Assumption~\ref{assumption_main} hold, let $y \in Y$ satisfy \eqref{cond_Picard}, $\yd \in Y$ satisfy \eqref{cond_noise}, and let $\AD y$ and $\zad$ be as in \eqref{def_AD_expl} and \eqref{def_zad}, respectively. Moreover, let $\mu > 0$, $\alpha_0 > 0$, and assume that for all $\alpha \in (0,\alpha_0)$ and $\lambda \in \C$ the function $r_\alpha$ defined in \eqref{def_ra} satisfies 
    \begin{equation}\label{cond_ra_plus}
		\abs{\lambda}^{\mu +1/2} \abs{r_\alpha(\lambda)} \leq c_{\mu+1/2} \alpha^{\mu+1/2} \,,
	\end{equation}
for some constant $c_{\mu+1/2} > 0$. Furthermore, let \eqref{cond_Ga_bound} and the source condition \eqref{sourcecond_frames} hold, and let the function $\alpha \mapsto g_\alpha(\lambda)$ be continuous from the left for all $\lambda \in \C$. Then for $\alpha = \alpha(\delta,\yd)$ chosen via the a-posteriori stopping rule \eqref{tau}, \eqref{discrepancy_FD}, it follows that
    \begin{equation*}
        \norm{\AD y - \zad}_X = \LandauO\kl{ \delta^{\frac{2\mu}{2\mu + 1}}  } \,.
    \end{equation*}
\end{theorem}
\begin{proof}
First, assume that there are sequences $\delta_n \to 0$ and $y^{\delta_n} \in Y$ satisfying
    \begin{equation*}
        \norm{y-y^{\delta_n}}_Y \leq \delta_n \,,
    \end{equation*}
such that for all $n$ there holds $\alpha_n:= \alpha(\delta_n,y^{\delta_n}) = \infty$. Then due to \eqref{discrepancy_FD} there holds
    \begin{equation*}
        \sum_\lknz^\infty  \abs{r_{\alpha}(\lk^2) \spr{y^\dn,f_k}_Y }^2
        \leq
        (\tau \dn)^2  \,,
    \end{equation*}
for all $\alpha > 0$.  
    Hence, together with \eqref{def_frame} and the definition \eqref{def_gamma} of $\gamma$ we obtain
    \begin{equation}\label{helper_sum_p1}
    \begin{aligned}
        \sum_\lknz^\infty  \abs{r_\alpha(\lk^2) \spr{y,f_k}_Y }^2 
        &\leq
        2\sum_\lknz^\infty  \abs{r_\alpha(\lk^2) \spr{y-y^\dn,f_k}_Y }^2 
        +
        2\sum_\lknz^\infty  \abs{r_\alpha(\lk^2) \spr{y^\dn,f_k}_Y }^2
        \\
        &\leq
        2\gamma^2 C_2 \norm{y-y^\dn}_Y^2
        +
        2(\tau \dn )^2 
        \leq
        2( \gamma^2 C_2 + \tau^2) \delta_n^2
         \,.
    \end{aligned}
    \end{equation}
Letting $n\to \infty$ in the above inequality we thus obtain that for all $\alpha > 0$ there holds
    \begin{equation*}
        \sum_\lknz^\infty  \abs{r_\alpha(\lk^2) \spr{y,f_k}_Y }^2 = 0 \,.
    \end{equation*}
Hence, using the dominated convergence theorem we obtain
    \begin{equation*}
        0 
        = 
        \lim_{\alpha \to \infty} \sum_\lknz^\infty  \abs{r_\alpha(\lk^2) \spr{y,f_k}_Y }^2 
        =
        \sum_\lknz^\infty \lim_{\alpha \to \infty} \abs{r_\alpha(\lk^2)}^2 \abs{ \spr{y,f_k}_Y }^2 \,.
    \end{equation*}
Now since from the definition \eqref{def_Ga} of $G_\alpha$ and \eqref{cond_Ga_bound} there follows
    \begin{equation*}
        \abs{r_\alpha(\lk^2)} 
        =
        \abs{1-\lk^2 g_\alpha(\lk^2)}
        \geq
        1 - \abs{\lk^2 g_\alpha(\lk^2)}
        \overset{\eqref{def_Ga}}{\geq}
        1- \abs{\lk^2} G_\alpha 
        \overset{\eqref{cond_Ga_bound}}{\geq}
        1 - \abs{\lk^2} \hat{c}/\alpha \,,
    \end{equation*}
we obtain that $\lim_{\alpha \to \infty} \abs{r_\alpha(\lk^2)} \geq 1$ and thus we find that
    \begin{equation*}
        0 
        = 
        \sum_\lknz^\infty \lim_{\alpha \to \infty} \abs{r_\alpha(\lk^2)}^2 \abs{ \spr{y,f_k}_Y }^2
        \geq
        \sum_\lknz^\infty  \abs{ \spr{y,f_k}_Y }^2
        \,.
    \end{equation*}
Since this implies $\spr{y ,f_k}_Y = 0$ for all $\lk\neq 0$, it follows from the definition of $\AD y$ that
    \begin{equation*}
        \AD y = 0 = z_{\alpha_n}^\dn \,.
    \end{equation*}
Hence, for the remainder of this proof we can assume that $\alpha(\delta,\yd) < \infty$ for all $\yd$ satisfying \eqref{cond_noise} with $\delta$ sufficiently small. Next, note that in \eqref{helper_ineq_1} we have shown that
    \begin{equation*}
		\norm{\AD y - \za}_Y^2 
		\leq \frac{1}{B_1}
		\sum_\lknz^\infty \abs{\frac{r_\alpha(\lk^2) }{\lk} \spr{y,f_k}_Y }^2 \,.
	\end{equation*}
Since by the H\"older inequality with $p=2\mu+1$ and $q=(2\mu+1)/(2\mu)$ there follows
    \begin{equation*}
    \begin{split}
        &\sum_\lknz^\infty \abs{r_\alpha(\lk^2) \lk^{-1} \spr{y,f_k}_Y }^2
        =
        \sum_\lknz^\infty \abs{r_\alpha(\lk^2) \lk^{-(2\mu+1)}\spr{y,f_k}_Y}^{ \frac{2}{2\mu+1}}
        \abs{r_\alpha(\lk^2) \spr{y,f_k}_Y }^{\frac{4\mu}{2\mu+1}}
        \\
        & \qquad 
        \leq
        \kl{\sum_\lknz^\infty 
        \abs{r_\alpha(\lk^2) \lk^{-(2\mu+1)}\spr{y,f_k}_Y}^{2} }^{\frac{1}{2\mu+1}}
        \kl{\sum_\lknz^\infty 
        \abs{r_\alpha(\lk^2) \spr{y,f_k}_Y }^{2}}^{\frac{2\mu}{2\mu+1}} \,,
    \end{split}
    \end{equation*}
we thus obtain that    
    \begin{equation}\label{helper_sum_1}
		\norm{\AD y - \za}_Y^2 
		\leq
		\frac{1}{B_1}
		\kl{\sum_\lknz^\infty 
        \abs{r_\alpha(\lk^2) \lk^{-(2\mu+1)}\spr{y,f_k}_Y}^{2} }^{\frac{1}{2\mu+1}}
        \kl{\sum_\lknz^\infty 
        \abs{r_\alpha(\lk^2) \spr{y,f_k}_Y }^{2}}^{\frac{2\mu}{2\mu+1}} \,.
	\end{equation}
We now consider each of these sums separately, where we have already shown in \eqref{helper_sum_p1} that the second sum in \eqref{helper_sum_1} can be estimated by
 \begin{equation*}
        \sum_\lknz^\infty  \abs{r_\alpha(\lk^2) \spr{y,f_k}_Y }^2 
        \leq
        2( \gamma^2 C_2 + \tau^2) \delta^2
         \,.
    \end{equation*}  
Concerning the first sum in \eqref{helper_sum_1}, note that due to the source condition \eqref{sourcecond_frames},
    \begin{equation*}
        \sum_\lknz^\infty 
        \abs{r_\alpha(\lk^2) \lk^{-(2\mu+1)}\spr{y,f_k}_Y}^{2} 
        =
        \sum_\lknz^\infty 
        \abs{r_\alpha(\lk^2) \spr{w,e_k}_X }^{2} \,.
    \end{equation*}
Together with the definition of $\gamma$ and since $\{e_k \}_{k\in\N}$ forms a frame, we obtain
    \begin{equation}\label{helper_sum_p2}
        \sum_\lknz^\infty 
        \abs{r_\alpha(\lk^2) \lk^{-(2\mu+1)}\spr{y,f_k}_Y}^{2} 
        =
        \sum_\lknz^\infty 
        \abs{r_\alpha(\lk^2) \spr{w,e_k}_X }^{2}
        \leq 
        B_2 \gamma^2 \norm{w}_X^2 \,.
    \end{equation}  
Hence, inserting \eqref{helper_sum_p1} and \eqref{helper_sum_p2} into \eqref{helper_sum_1} we obtain
    \begin{equation*}
        \norm{\AD y - \za}_Y^2 
		\leq
		\frac{1}{B_1}
		\kl{B_2 \gamma^2 \norm{w}_X^2 }^{\frac{1}{2\mu+1}}
        \kl{2(C_2\gamma^2 + \tau^2)\delta^2}^{\frac{2\mu}{2\mu+1}} \,,
    \end{equation*}
which simplifies to 
    \begin{equation}\label{helper_est_exact}
        \norm{\AD y - \za}_Y
		\leq
		\kl{ \frac{1}{B_1}		\kl{ B_2 \gamma^2 \norm{w}_X^2 }^{\frac{1}{2\mu+1}}
        \kl{2(C_2\gamma^2 + \tau^2)}^{\frac{2\mu}{2\mu+1}} }^{1/2}
        \delta^{\frac{2\mu}{2\mu+1}}
        \,,
    \end{equation}
and thus provides an (optimal) convergence rate given exact data. Next, note that due to the reverse triangle inequality there holds
    \begin{equation}\label{help_sum_revtriangle}
    \begin{split}
        \kl{\sum_\lknz^\infty  \abs{r_{2\alpha}(\lk^2) \spr{y-\yd,f_k}_Y }^2}^{1/2}
        &\geq
        \kl{\sum_\lknz^\infty  \abs{r_{2\alpha}(\lk^2) \spr{\yd,f_k}_Y }^2}^{1/2}\\
        &\quad-
        \kl{\sum_\lknz^\infty  \abs{r_{2\alpha}(\lk^2) \spr{y,f_k}_Y }^2}^{1/2}.
    \end{split}
    \end{equation}
By the definition \eqref{discrepancy_FD} of our a-posteriori parameter choice rule we obtain
    \begin{equation*}
        \sum_\lknz^\infty  \abs{r_{2\alpha}(\lk^2) \spr{\yd,f_k}_Y }^2 
        > (\tau \delta)^2 \,.
    \end{equation*}
Thus, rearranging \eqref{help_sum_revtriangle}, together with the definition of $\gamma$, \eqref{cond_noise}, and \eqref{eq_bound_F_Fadj} we obtain
    \begin{equation*}
    \begin{split}
        \kl{\sum_\lknz^\infty  \abs{r_{2\alpha}(\lk^2) \spr{y,f_k}_Y }^2 }^{1/2}
        &\geq 
        (\tau \delta) - \gamma \kl{\sum_\lknz^\infty  \abs{\spr{y-\yd,f_k}_Y }^2}^{1/2}
        \\
        &\geq
        (\tau \delta) - \gamma \sqrt{C_2}\norm{y-\yd}_Y
        \geq 
        \kl{\tau - \gamma \sqrt{C_2}} \delta \,.
    \end{split}
    \end{equation*}
Now since due to \eqref{tau} there holds
    \begin{equation*}
        c:= \tau - \gamma \sqrt{C_2} > 0 \,,
    \end{equation*}
it follows that
    \begin{equation*}
        (c \delta)^2 
        \leq 
        \sum_\lknz^\infty  \abs{r_{2\alpha}(\lk^2) \spr{y,f_k}_Y }^2
        =
        \sum_\lknz^\infty  \abs{r_{2\alpha}(\lk^2) \lk^{2\mu+1}}^2 \abs{ \lk^{-(2\mu+1)}\spr{y,f_k}_Y }^2 \,.
    \end{equation*}
Since due to \eqref{cond_ra_plus} and the definition of $\gamma$ there holds
    \begin{equation*}
        \abs{r_{2\alpha}(\lk^2) \lk^{2\mu+1}}^2 
        \leq 
        \kl{c_{\mu+1/2} (2\alpha)^{\mu+1/2}}^2
        =
        c_{\mu+1/2}^2 (2\alpha)^{2\mu+1} \,,
    \end{equation*}
it follows together with the source condition \eqref{sourcecond_frames} that
    \begin{equation*}
        (c \delta)^2 
        \leq 
        c_{\mu+1/2}^2 (2\alpha)^{2\mu+1}
        \sum_\lknz^\infty  \abs{ \lk^{-(2\mu+1)}\spr{y,f_k}_Y }^2
        =
        c_{\mu+1/2}^2 (2\alpha)^{2\mu+1}
        \sum_\lknz^\infty  \abs{ \spr{w,e_k}_X }^2
        \,.
    \end{equation*}
Hence, together with \eqref{def_frame} we obtain
    \begin{equation*}
        (c \delta)^2 \leq c_{\mu+1/2}^2 (2\alpha)^{2\mu+1} B_2 \norm{w}_X^2 \,,
    \end{equation*}
and thus we find that
    \begin{equation}\label{helper_est_alpha}
        \alpha^{-1}   \leq 2 \kl{c_{\mu+1/2} B_2^{1/2} \norm{w}_X /c}^{\frac{2}{2\mu+1}} \delta^{-\frac{2}{2\mu+1}}  \,.
    \end{equation}
Now, note that due to \eqref{helper_ineq_2} and \eqref{cond_Ga_bound} there holds
    \begin{equation}\label{helper_LandauO_1}
        \norm{\AD y - \zad}_Y 
        \overset{\eqref{helper_ineq_2}}{\leq} \norm{\AD y - \za}_X + \delta \sqrt{C G_\alpha C_2/B_1}
        \overset{\eqref{cond_Ga_bound}}{\leq}
        \norm{\AD y - \za}_X + \frac{\delta}{\sqrt{\alpha}} \sqrt{\hat{c}\,C C_2/B_1}  
         \,.
    \end{equation}    
Inserting \eqref{helper_est_exact} and \eqref{helper_est_alpha} into this inequality yields
    \begin{equation*}
    \begin{split}
        \norm{\AD y - \zad}_Y 
        &\leq
        \kl{ \frac{1}{B_1}		\kl{ B_2 \gamma^2 \norm{w}_X^2 }^{\frac{1}{2\mu+1}}
        \kl{2(C_2\gamma^2 + \tau^2)}^{\frac{2\mu}{2\mu+1}} }^{1/2}
        \delta^{\frac{2\mu}{2\mu+1}}
        \\
        &+ 
        \kl{ \sqrt{2} \kl{c_{\mu+1/2} B_2^{1/2} \norm{w}_X /c}^{\frac{1}{2\mu+1}} \sqrt{\hat{c}\,C C_2/B_1}  } \delta \delta^{-\frac{1}{2\mu+1}} 
    \end{split}\,,
    \end{equation*} 
which now yields the assertion.
\end{proof}

% % % % % % % % % % % % % % % % % % % % % % % % %
% Section - Application to the Radon Transform  %
% % % % % % % % % % % % % % % % % % % % % % % % %
\section{Application to the Radon Transform}\label{sect_Radon_Trans}

In this section, we illustrate our theoretical results on continuous regularization via FDs by applying them to a standard tomography problem based on the \emph{Radon transform} \cite{Natterer_2001,Louis_1989}, which in 2D is given by
    \begin{equation}\label{Radon_A}
    \begin{split}
        (Ax)(s,\vphi) :=  
        \int_\R x(s\omega(\vphi) + t \omega(\vphi)^\perp) \, dt \,,
    \end{split}	
    \end{equation}
where $\omega(\vphi) = (\cos(\vphi),\sin(\vphi))^T$ for $\vphi \in [0,2\pi)$ and $s \in \R$. After recalling some recently derived FDs of the Radon transform on which we base the numerical illustration of our theoretical results derived above, we consider the details of the implementation of these decompositions and provide a number of numerical examples. For comparison, we also present numerical results based on the SVD of the Radon transform; cf.~\cite{Louis_1989,Natterer_2001}.

% % % % % % % % % % % % % % % % % % % % % % % % % % % % % % %
% Subsection - Frame Decompositions of the Radon Transform  %
% % % % % % % % % % % % % % % % % % % % % % % % % % % % % % %
\subsection{Frame Decompositions of the Radon Transform}\label{sect_Radon_Frames}

A number of different decompositions of the Radon transform fitting into the category of frame decompositions have been studied in the past. These include e.g.\ the WVD \cite{Donoho_1995} or the the biorthogonal curvelet/shearlet decompositions \cite{Candes_Donoho_2002,Colonna_Easley_Guo_Labate_2010}, for which also efficient implementations are available. For the numerical examples presented in this paper, which mainly serve to illustrate the theoretical results derived above, we focus on a class of FDs recently introduced in \cite{Hubmer_Ramlau_2021_01}. These are based on a general strategy for deriving FDs for arbitrary bounded linear operators in Hilbert spaces satisfying a stability condition of the form \eqref{cond_A_stability}. In contrast to the classic SVD, available for the setting
    \begin{equation*}
         A: \LtOD\to L_2(\Omega_S,w^{-1}) \,,
         \quad 
         w(s)=\sqrt{1-s^2}\,,
     \end{equation*}
where $\OD :=  \{x \in \R^2 \, \vert \, \abs{x} \leq 1\}$ and $\OS := \R \times [0,2\pi)$, these FDs of the Radon transform are also available for the very general settings
    \begin{equation}\label{Radon_Sobolev}
        A : \HbzOD \to \LtOS \,,
        \qquad
        \text{and}
        \qquad
        A : \HbzOD \to \LtOSb \,,
    \end{equation}
where $\OSb := [-1,1]\times [0,2\pi)$. While we refer to \cite{Hubmer_Ramlau_2021_01} for the most general version of these FDs, here we focus only on two special cases based on wavelets and exponentials.

\begin{theorem}\label{thm_Radon_wavelets}\cite[Thm.~5.2]{Hubmer_Ramlau_2021_01}
Let $0 \leq \beta \in \R$ and let $A : \HbzOD \to \LtOS$ be the Radon transform as defined in \eqref{Radon_A}. Furthermore, let $\{\psi_{j,k}\}_{j,k\in\Z}$ be an orthonormal wavelet basis corresponding to an $r$-regular multiresolution analysis of $\LtR$ with $r > \beta$, let $\{w_l\}_{l\in\N}$ be an orthonormal basis of $\LtT$, and define 
    \begin{equation}\label{wavframes}
	    \fjkl(s,\vphi) := \psi_{j,k}(s) w_l(\vphi) \,,
	    \qquad
	    \text{and}
	    \qquad
	    \ejkl := \kl{1 + 2^{-2j \kl{\beta+1/2}}}^{1/2}A^*\fjkl \,.
	\end{equation}
Then the set $\{\ejkl\}_{j,k\in \Z\,, l \in \N}$ forms a frame over $\HbzOD$ and 
	\begin{equation*}
	    A x = \sum\limits_{j,k \in \Z}\sum\limits_{l=1}^\infty \kl{1 + 2^{-2j \kl{\beta+1/2}}}^{-1/2} \spr{x, \ejkl}_\HbzOD \fjklt 
	    =
	    \sum\limits_{j,k \in \Z}\sum\limits_{l=1}^\infty \spr{x,A^*\fjkl} \fjklt \,.
	\end{equation*}
Furthermore, for any $y \in R(A)$ the unique solution of $Ax=y$ is given by 
	\begin{equation}\label{dec_ADy_Radon_specific}
	    \AD y = \sum\limits_{j,k \in \Z} \sum\limits_{l=1}^\infty \kl{1 + 2^{-2j \kl{\beta+1/2}}}^{1/2}\spr{y,\fjkl}_\LtOS \ejklt \,.
	\end{equation}
\end{theorem}

A possible choice for the orthonormal basis $\Kl{w_l}_{l\in\N}$ in the above theorem is e.g.,
    \begin{equation*}
         w_l(\vphi):= (1/\sqrt{2\pi}) \exp{(il\vphi)}\,,
    \end{equation*}
which was also used throughout the numerical experiments presented below. 

\begin{theorem}\label{thm_Radon_exp}\cite[Remark~5.2]{Hubmer_Ramlau_2021_01}
Let $0 \leq \beta \in \R$ and let $A : \HbzOD \to \LtOSb$ be the Radon transform as defined in \eqref{Radon_A}. Furthermore, let 
	\begin{equation}\label{expframes}
	    w_{j,k}(s,\vphi) := \frac{1}{2\sqrt{ \pi}}\exp(i j \pi s) \exp(ik \vphi) \,,
	    \qquad
	    \text{and}
	    \qquad
	    v_{j,k} := \kl{1+\abs{j}^2}^{\kl{\beta+1/2}/2} A^* w_{j,k} \,.
	\end{equation} 
Then the set $\{v_{j,k}\}_{j,k\in \Z}$ forms a frame over $\HbzOD$ and 
    \begin{equation*}
        A x = \sum\limits_{j,k \in \Z} \kl{1+\abs{j}^2}^{-\kl{\beta+1/2}/2} \spr{x, v_{j,k}}_\HbzOD \tilde{w}_{j,k} 
	    =
	    \sum\limits_{j,k \in \Z} \spr{x,A^* w_{j,k}} \tilde{w}_{j,k} \,.
    \end{equation*}
Furthermore, for any $y\in R(A)$ the unique solution of $Ax=y$ is given by
	\begin{equation}\label{dec_ADy_Radon_wjk}
	    \AD y = \sum\limits_{j,k\in\Z} \kl{1+\abs{j}^2}^{\kl{\beta+1/2}/2} \spr{y,w_{j,k}}_\LtOSb \tilde{v}_{j,k} \,.
	\end{equation}
\end{theorem}

Note that the FDs and the corresponding formulas for computing $\AD y$ given in Theorem~\ref{thm_Radon_wavelets} and \ref{thm_Radon_exp} can be efficiently implemented by approximating the involved inner products via fast Fourier and wavelet transforms. While this can entail a loss of accuracy compared to higher order integration methods, the increased efficiency is beneficial for applications with real-time requirements such as atmospheric tomography \cite{Hubmer_Ramlau_2021_01}.

% % % % % % % % % % % % % % % % % % % % % % % % % % % % %
% Subsection - Implementation and Computational Aspects %
% % % % % % % % % % % % % % % % % % % % % % % % % % % % %
\subsection{Implementation and Computational Aspects}\label{subsect_implementation}

In this section, we consider the implementation of the FDs of the Radon transform presented in Section~\ref{sect_Radon_Frames}, focusing on the relevant special case $\beta = 0$ in \eqref{Radon_Sobolev}, i.e.,
    \begin{equation*}
        A : \LtOD \to \LtOS \,,
        \qquad
        \text{and}
        \qquad
        A : \LtOD \to \LtOSb \,.
    \end{equation*}
The key difference in implementation between the FDs given in Theorem~\ref{thm_Radon_wavelets} and Theorem~\ref{thm_Radon_exp} lies in the different sets of frame functions and their corresponding properties. Note that if the dual frame functions $\ejklt$ or $\tilde{v}_{j,k}$ are pre-computed and stored, then for each right-hand side $y$ the computation of $\AD y$ amounts only to the computation of either the inner products $\spr{y,f_{j,k,l}}_{\LtOS}$ or $\spr{y,w_{j,k}}_\LtOSb$, and a corresponding summation according to \eqref{dec_ADy_Radon_specific} or \eqref{dec_ADy_Radon_wjk}, respectively. As mentioned above, these inner products can be efficiently implemented using fast Fourier and wavelet transforms.

% % % % % % % % % % % % % % % % % % % % % % % % % % % % % % % % %
% Subsubsection - Discretization and Computational Environment  %
% % % % % % % % % % % % % % % % % % % % % % % % % % % % % % % % %
\subsubsection{Problem Discretization and Computational Environment}

For the discretization of the problem we have used the AIR Tools II toolbox \cite{Hansen_2018}, which is based on a piecewise constant discretization of both the definition and the image space of $A$. More precisely, a density function $x \in \LtOD$ is approximated by a piecewise constant function with values given on a uniform $N \times N$ pixel grid. Similarly, sinogram data $y$ are also considered as piecewise constant functions on a uniform $p \times N_\theta$ pixel grid, where $p$ denotes the number of equidistant, parallel lines on $[-1,1]$, and $N_\theta$ denotes the number of different angles $\theta_n$. Hence, the discretized problem can be written as 
    \begin{equation*}
        \Af  \xf = \yf \,,
        \qquad 
        \xf \in \R^{N^2}\,,\, 
        \yf \in \R^{pN_\theta}\,, \, 
        \Af \in \R^{N^2\times pN_\theta} \,.
    \end{equation*}
For all tests presented below, we have used the choice $N=p=60$, and $N_\theta=180$, with uniformly spaced angles $\theta_n= n \pi/180$ for $n=0\,,\dots\,,179$, which due to  $Ax(s,\vphi)=Ax(-s,\vphi+\pi)$ amounts to a full-angle tomography problem. The infinite sums in formulas \eqref{dec_ADy_Radon_specific} and \eqref{dec_ADy_Radon_wjk} for computing $\AD y$ have been replaced by finite sums as detailed below, and the involved integrals have been approximated using the trapezoidal rule. All computations have been performed using Matlab 2019b on a desktop computer running on Windows 10 with a 4 core processor
(Intel Core i5-6500 CPU@3.20GHz) and 16GB RAM, except the computation of the dual frames, which is discussed in detail below, and which was performed on the high performance computing cluster Radon1 \cite{Radon1specs}.

% % % % % % % % % % % % % % % % % % % % % % % % % % % % % %
% Subsubsection - Computation of the Dual Frame Functions %
% % % % % % % % % % % % % % % % % % % % % % % % % % % % % %
\subsubsection{Computation of the Dual Frame Functions} 

Next, we consider the computation of the dual frame functions $\ekt$. While it is generally not possible to give explicit expressions, one can use the recursive approximation \cite{Daubechies_1992}
    \begin{equation}\label{duals_iter}
        \tilde{e}_k\approx\tilde{e}_k^M=\frac{2}{B_1+B_2}e_k+ \kl{I - \frac{2}{B_1+B_2} S}\tilde{e}_k^{M-1},
    \end{equation}
where $B_1,B_2$ denote the frame bound of $\{e_k\}_{k\in\N}$ and $S$ is as in \eqref{Def_S}. The approximation error can be estimated by \cite{Daubechies_1992}: 
    \begin{equation}\label{eq_dual_approx_error}
        \norm{x - \sum\limits_{k=1}^\infty \spr{x,e_k}\et_k^M}_X \leq \kl{\frac{B_2-B_1}{B_2+B_1}}^{M+1}\norm{x}_X\,.
    \end{equation}
Hence, if the frame bounds $B_1$ and $B_2$ are close to each other only few iterations are necessary for obtaining an accurate approximation. Unfortunately, for the frames $\{e_{j,k,l}\}_{j,k\in\Z,l\in\N}$ and $\{v_{j,k}\}_{j,k\in\Z}$ in Theorem~\ref{thm_Radon_wavelets} and Theorem~\ref{thm_Radon_exp}, respectively, the frame bounds depend on constants of a norm-equality estimate like \eqref{cond_norm_Z} which are not known exactly. Hence, for computing an approximation of the dual frame functions via \eqref{duals_iter} a value for $B_1+B_2$ has to be chosen empirically. On the one hand, one has to choose this value large enough such that the recursion converges, while on the other hand a too large value results in a very slow convergence. For our numerical experiments, we found a value of $B_1+B_2=500$ to lead to satisfactory results within $M=50$ iterations.

Alternatively, one can compute the dual frame functions $\ekt$ directly via their definition $\ekt = S^{-1} e_k$. Discretizing as above, one has to solve the system of equations 
    \begin{equation*}
        \Sf \ektf= \ekf\,,
        \qquad
        \ektf \in \R^{N^2}\,,\, 
        \ekf \in \R^{N^2}\,, \, 
        \Sf \in \R^{N^2\times N^2} \,,
    \end{equation*}
where $\ektf$ and $\ekf$ are piecewise constant discretizations of $\ekt$ and $e_k$, respectively, and
     \begin{equation*}
        \Sf_{ij}=\left(\sum_{k=1}^{p N_\theta}\spr{\phi_i,e_k}\ekf\right)_j \,,
    \end{equation*}
where $\Kl{\phi_i}_{i\in N^2}$ denotes the piecewise constant pixel basis on $\OD$ used for discretization.
  
\begin{figure}[ht!]
    \centering
    \includegraphics[width=\textwidth]{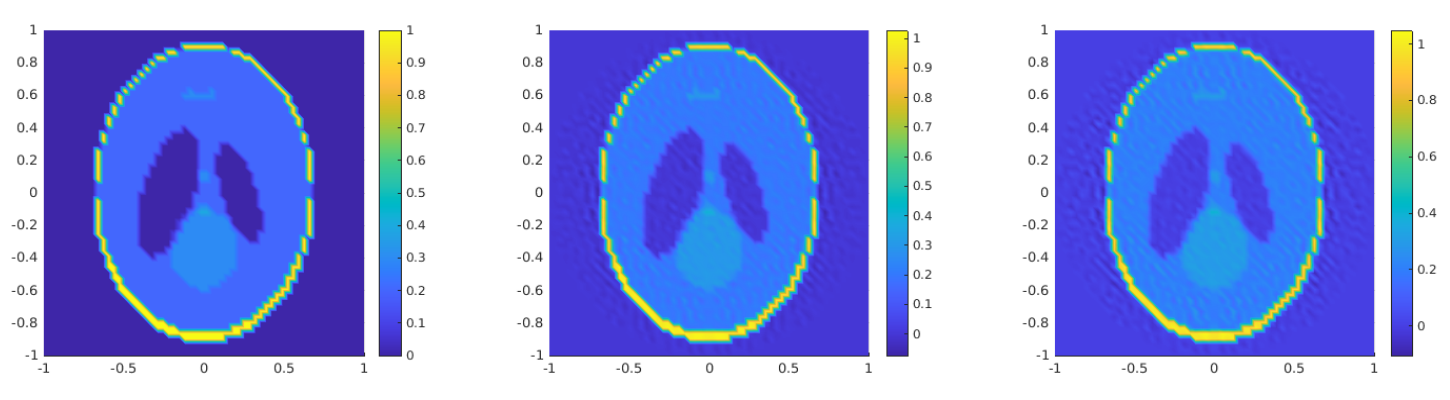}
    \caption{Shepp-Logan phantom (left) and its expansion \eqref{eq_frame_rec} in terms of the frame $\Kl{v_{j,k}}_{j,k\in\Z}$, with the functions $v_{j,k}$ as defined in \eqref{expframes} using the explicitly (middle) and the recursively (right) computed dual frames. Note that these computations are based on a coarser discretization with $N_\theta = 90$ instead of $N_\theta = 180$ as explained in the text.}
    \label{fig_reconstruction}
\end{figure}

Using an LU-decomposition for the matrix $\Sf$, the equation $\Sf \ektf =\ekf$ can be solved efficiently for all $k$. However, our matrices $\Sf$ are ill-conditioned, with condition-numbers
    \begin{equation*}
        \kappa_{wav}=\frac{\lambda_{max}}{\lambda_{min}} \approx 7.7\cdot 10^5 \,,
        \qquad \text{and} \qquad 
        \kappa_{exp}=\frac{\lambda_{max}}{\lambda_{min}} \approx 1.5\cdot 10^5 \,, 
    \end{equation*}
where $\lambda_{max}$ and $\lambda_{min}$ denote the maximal and minimal singular value of $\Sf$, respectively. Hence, to obtain stable approximations of $\ektf$ we used Tikhonov regularization, i.e., 
    \begin{equation}\label{duals_explicit}
        \ektf \approx (\Sf^T \Sf +\alpha I)^{-1} \ekf \,.
    \end{equation}
We found optimal results for $\alpha=0.01$ for the wavelet-based FD and $\alpha=2$ for the exponential-based FD. For these particular choices of $\alpha$, this explicit approach outperformed the recursive approximation. The computed dual frame functions were verified by implementing the reconstruction formula \eqref{eq_frame_rec}, an example of which is shown in Figure~\ref{fig_reconstruction}. Note that for obtaining these results, a coarser discretization of $N_\theta=90$ had to be used in the numerical computation of the frame functions $v_{j,k}$ defined via \eqref{expframes}, since the computation of the recursively approximated dual frames $\tilde{v}_{j,k}$ already takes about $8$ hours in this setup, while the full setup with $N_\theta = 180$ angles is computationally infeasible. The reconstructions shown in Figure~\ref{fig_reconstruction} have a relative error of $6.98\%$ in the case of explicitly computed dual frames, and an error of $9.07\%$ in the case of recursively approximated dual frames. Thus, not only is the explicit computation approach faster than the recursive approach, but it also outperforms it in terms of reconstruction quality. Hence, all numerical results presented below are based on the explicitly computed dual frame functions using the full setup with $N_\theta = 180$ angels.

% % % % % % % % % % % % % % % % % % % % % % % % % % % %
% Subsubsection - Implementation of Wavelet-based FD  %
% % % % % % % % % % % % % % % % % % % % % % % % % % % %
\subsubsection{Implementation of the Wavelet-based Frame Decomposition}

\begin{figure}[ht!]
    \centering
    \includegraphics[width=\textwidth]{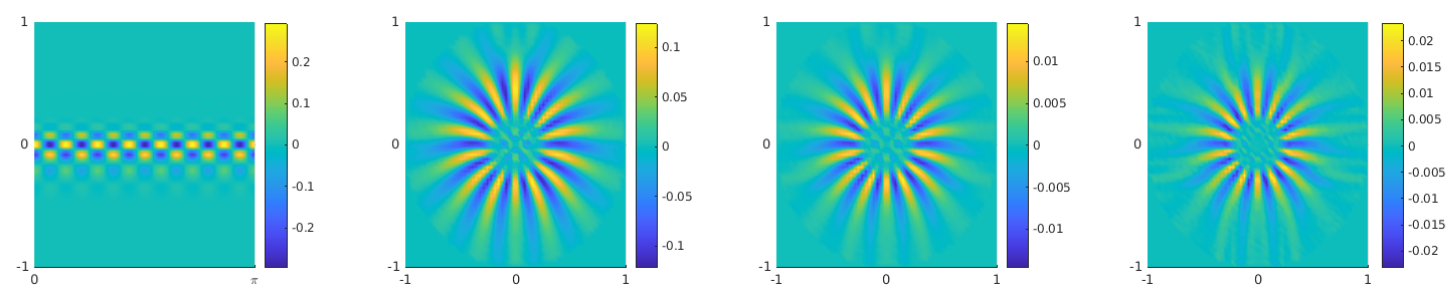}
    \includegraphics[width=\textwidth]{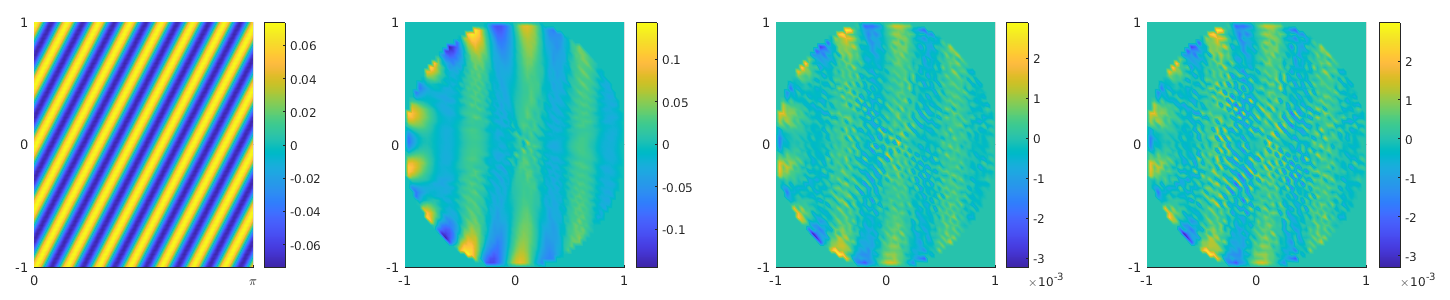}
    \caption{Example frame functions based on wavelets (top, $j=2$, $k=4$, $l=7$) and exponential functions (bottom, $j=4$, $k=7$). From left to right: $f_{k}$, $e_{k}$, $\ekt^M$ computed with the recursive formula \eqref{duals_iter} and $\ekt$ computed explicitly via \eqref{duals_explicit}.}
    \label{fig_frames}
\end{figure}

We now consider the setup from Theorem~\ref{thm_Radon_wavelets}. In particular, for $\Kl{\psi_{j,k}}_{j,k\in\Z}$ we consider inhomogeneous orthonormal wavelet bases of the form $\{\psi^1_{j_{max},k}\}_{k\in\Z} \cap \{\psi^2_{j,k}\}_{j\leq j_{max},k\in \Z}$. These are defined via $\psi^1_{j,k}(s)=2^{-j/2} \psi^1(2^{-j} s - k)$ and $\psi^2_{j,k}(s)=2^{-j/2} \psi^2(2^{-j} s - k)$ based on suitable scaling- and wavelet function $\psi^1$ and $\psi^2$, respectively (cf.~\cite{Daubechies_1992} for details). A typical example of such wavelet bases are the different Daubechies wavelets \cite{Daubechies_1992}, for which the  corresponding scaling- and wavelet functions are available in Matlab. Using such an inhomogeneous wavelet basis, any $x\in L_2(\R)$ can be written in the form \cite{Daubechies_1992}:
    \begin{equation*}
        x=\sum_{k\in\Z} \spr{x,\psi_{j_{max},k}^1}_\LtR \psi_{j_{max},k}^1+\sum_{j=-\infty}^{j_{max}}\sum_{k\in\Z} \spr{x,\psi_{j,k}^2}_{\LtR}\psi_{j,k}^2 \,.
    \end{equation*}
In our implementation, we used the `db4' Daubechies-wavelet basis \cite{Daubechies_1992}, characterized by having $4$ vanishing moments. Using this setting, the expression \eqref{dec_ADy_Radon_specific} for $\AD y$ becomes 
    \begin{equation}\label{num1_FDwavexact}
        \AD y= \sum\limits_{j=-\infty}^{j_{max}+1}
         \sum\limits_{k\in \Z}
         \sum\limits_{l=1}^\infty \frac{1}{\alpha_j}\spr{y,\fjkl}_\LtOS \ejklt \,,
    \end{equation}
where
    \begin{equation*}
        \alpha_j=\begin{cases}
            \kl{1 + 2^{-j_{max}}}^{-1/2}, &j=j_{max}+1 \,,
            \\
            \kl{1 + 2^{-j}}^{-1/2}, & j\leq j_{max} \,,
        \end{cases}
    \end{equation*}
and
    \begin{equation*}
        \fjkl(s,\theta)=
        \begin{cases}
            \psi^1_{j_{max},k}(s)\exp{(il\theta)}/\sqrt{2\pi} \,, \quad  &j= j_{max}+1 \,.
            \\
            \psi^2_{j,k}(s)\exp{(il\theta)}/\sqrt{2\pi} \,, \quad  & j\leq j_{max} \,. 
        \end{cases}
    \end{equation*}
For computation, we replaced the infinite sums in \eqref{num1_FDwavexact} by
    \begin{equation}\label{num2_FDwavexact}
        \AD y
        \approx 
        \sum\limits_{j=j_{min}}^{j_{max}+1}
         \sum\limits_{k\in K(j)}
         \sum\limits_{l=1}^L \frac{1}{\alpha_j}\spr{y,\fjkl}_\LtOS \ejklt 
    \end{equation}
where $K(j): =\{k\in\Z:\supp{\psi^1_{j,k}}\cap [-1,1]\}$ for $j\leq j_{max}$, $K(j_{max}+1) := K(j_{max})$. We used $j_{min}=0$, $j_{max}=3$, and $L=N_\theta$. Analogously, for the approximate solutions $\zad$ we use
    \begin{equation}\label{num_FDwavnoise}
    \begin{aligned}
        z_\alpha^\delta\approx &\sum\limits_{j=j_{min}}^{j_{max}+1}
         \sum\limits_{k\in K(j)}
         \sum\limits_{l=1}^L \alpha_j g_\alpha(\alpha_j^2)\spr{y,\fjkl}_\LtOS \ejklt  \,,
    \end{aligned}
    \end{equation}
where for the function $g_\alpha$ we use the filter functions of Tikhonov regularization, Landweber iteration, and the truncated SVD as defined in \eqref{filters}, respectively.

\begin{remark}
Note that since due to \eqref{cond_frames_connected} for all $x \in \LtOD$ there holds
    \begin{equation*}
        \alpha_j\spr{x,e_{j,k,l}}_\LtOD = \spr{Ax,f_{j,k,l}}_\LtOS \,,
    \end{equation*}
in the the computation of the dual frame functions $\ejklt$ the computation of the inner products $\spr{x,e_{j,k,l}}_\LtOD$ can be replaced by the computation of the inner products $\spr{Ax,f_{j,k,l}}_\LtOS$. This is advantageous, since the support of $f_{j,k,l}$ is known explicitly (cf.~Figure~\ref{fig_frames}) and thus the integration domain can be restricted. 
\end{remark}

% % % % % % % % % % % % % % % % % % % % % % % % % % % % % % 
% Subsubsection - Implementation of Exponential-based FD  %
% % % % % % % % % % % % % % % % % % % % % % % % % % % % % %
\subsubsection{Implementation of the Exponential-based Frame Decomposition}

For the setup as in Theorem~\ref{thm_Radon_exp} we replace the infinite sum in \eqref{dec_ADy_Radon_wjk} by
    \begin{equation}\label{num_FDexpexact}
	    \AD y \approx \sum\limits_{j=j_{min}}^{j_{max}}\sum\limits_{k=1}^{N_\theta} \frac{1}{\alpha_j}\spr{y,w_{j,k}}_\LtOS \tilde{v}_{j,k} \,,
    \end{equation}
using $j_{min}=-\lfloor p/2 \rfloor=-30$, $j_{max}=\lceil p/2 \rceil-1=29$, and $\alpha_j=\kl{1+\abs{j}^2}^{-1/4}$, which follows from our choice of $\beta = 0$. Analogously, for the approximate solutions $\zad$ we use
    \begin{equation}\label{num_FDexpnoise}
	    \zad \approx \sum\limits_{j=j_{min}}^{j_{max}}\sum\limits_{k=1}^{N_\theta} \alpha_j g_\alpha(\alpha_j^2)\spr{y,w_{j,k}}_\LtOS \tilde{v}_{j,k} \,,
    \end{equation}
with the filter functions of Tikhonov regularization and Landweber iteration as above.

% % % % % % % % % % % % % % % % % %
% Subsection - Numerical Results  %
% % % % % % % % % % % % % % % % % %   
\subsection{Numerical Results}\label{subsect_num_results}

In the following, we present some numerical results for the FDs of the Radon transform discussed above. The numerical experiments involve two regularization methods combined with an a-priori and an a-posteriori parameter choice rule for different noise levels. We also present equivalent results based on the SVD of the Radon transform, in order to compare the performance and behaviour of these methods. Note that since the SVD is a special case of the FD, the parameter choice rules discussed in Section~\ref{sect_Regularizations} can be used for both. The truncation value for the indices in the SVD-implementation was chosen such that the number of singular functions coincides with the number of frame functions in the exponential-based setup.

For our numerical experiments we use the Shepp-Logan phantom $x_{SL}$ depicted in Figure~\ref{fig_reconstruction} (left) as the exact solution, i.e., $\xD := x_{SL}$. For measuring the quality of the obtained reconstructions, we use both the relative $L_2$-error $\norm{z_\alpha^\delta-x_{SL}}/\norm{x_{SL}}$ and the structural similarity index measure (SSIM), which is defined in \cite{Wang_Bovik_Sheikh_Simoncelli_2004}. The SSIM is a value in $[0,1]$, with higher values indicating a stronger structural similarity. 

\begin{figure}[ht!]
    \centering
    \includegraphics[trim=0 40 0 0, clip,width=0.95\textwidth]{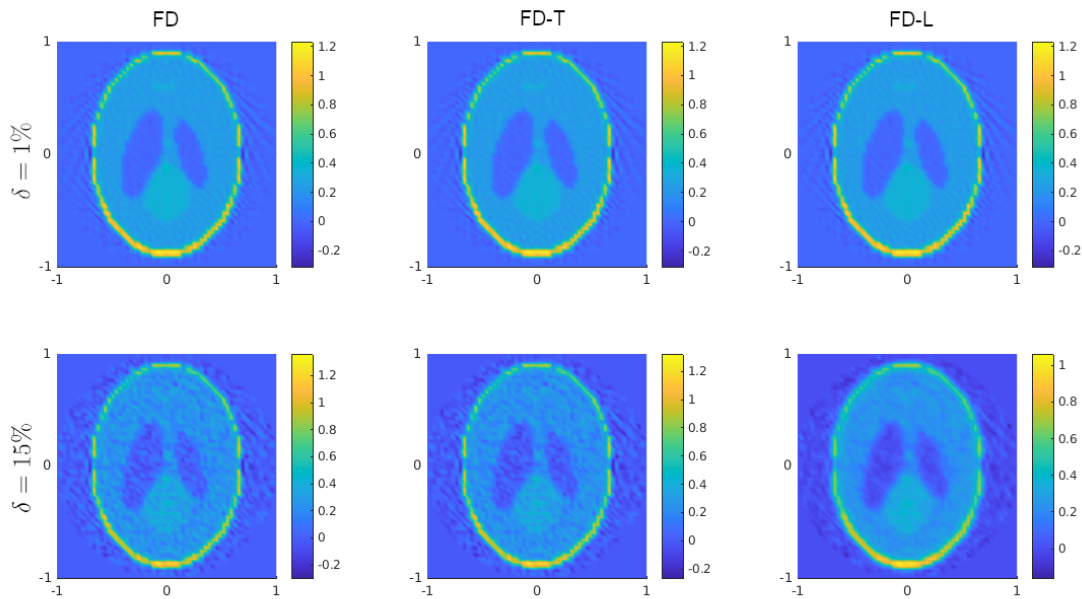}
    \caption{Reconstruction of the Shepp-Logan phantom using the SVD for the Radon transform using an a-priori parameter choice rule for different noise levels.}
    \label{fig_svdapriori}
\end{figure}

\begin{figure}[ht!]
    \centering
    \includegraphics[trim=0 40 0 0, clip,width=0.95\textwidth]{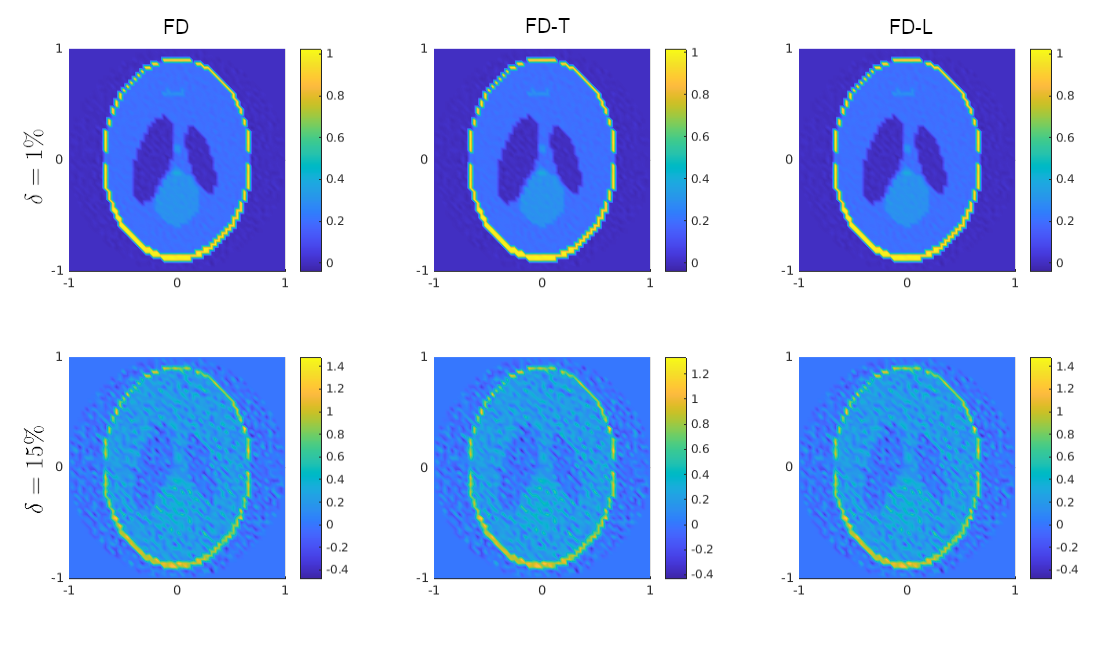}
    \caption{Reconstruction of the Shepp-Logan phantom using the FD for the Radon transform based on wavelet frames \eqref{num_FDwavnoise} using an a-priori parameter choice rule.}
    \label{fig_Wavapriori}
\end{figure}

\begin{figure}[ht!]
    \centering
    \includegraphics[trim=0 40 0 0, clip,width=0.95\textwidth]{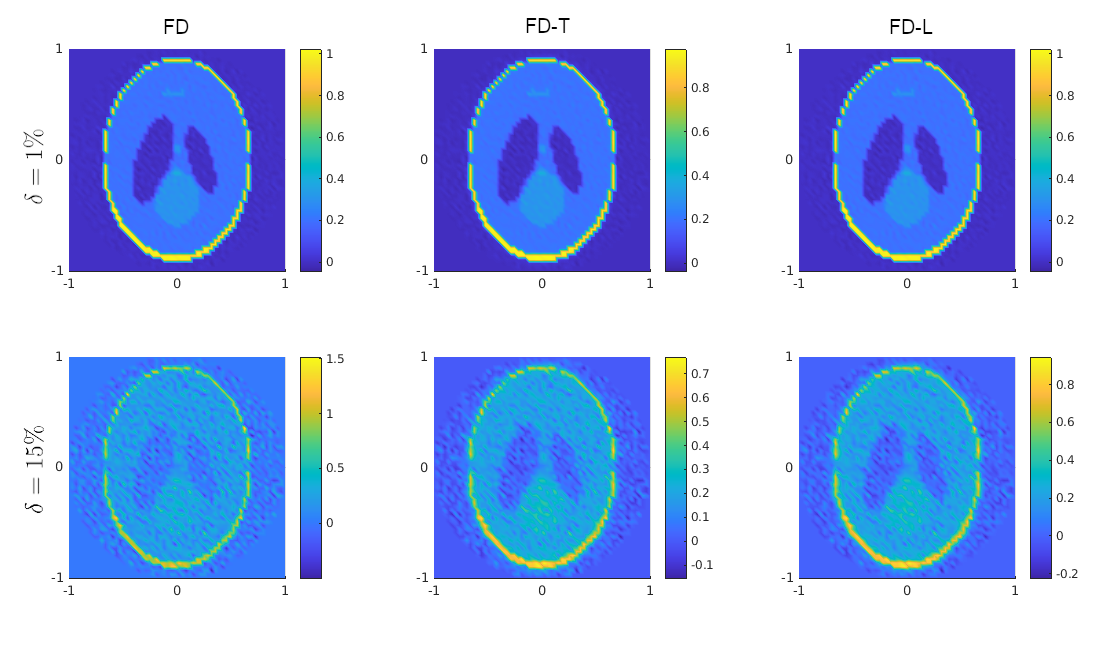}
    \caption{Reconstruction of the Shepp-Logan phantom using the FD for the Radon transform based on exponential frames \eqref{num_FDexpnoise} using an a-priori parameter choice rule.}
    \label{fig_expapriori}
\end{figure}

\begin{table}
    \centering
    \begin{tabular}{|c|r||r|r|r|r|r|}
    \hline
    \multicolumn{3}{|c|}{stopping rule:}&\multicolumn{2}{|c|}{\textbf{a-priori}}&\multicolumn{2}{|c|}{\textbf{a-posteriori}}
    \\ \hline
        \multicolumn{2}{|c|}{Filter function:} & none    & Tikh. & Landw.    & Tikh. & Landw.\\ \hline
    \multicolumn{2}{|c|}{}&\multicolumn{5}{|c|}{Relative $L_2$-error} \\ \hline
\hline
    \multirow{2}{*}{SVD}      & 1\% Noise       & 18.53\% & 18.53\%  & 18.54\%   & 18.53\%  & 18.56\%  \\ \cline{2-2} 
                              & 15\% Noise      & 20.91\% & 20.75\%  & 27.92\%   & 20.81\%  & 21.30\%  \\ \hline
    \multirow{2}{*}{Wav}      & 1\% Noise       &  2.42\% &  2.51\%  &  2.42\%   &  2.45\%  &  2.56\%  \\ \cline{2-2} 
                              & 15\% Noise      & 25.43\% & 23.00\%  & 25.43\%   & 24.56\%  & 24.41\%  \\ \hline
    \multirow{2}{*}{Exp}      & 1\% Noise       & 2.54\%  & 5.18\%   & 2.54\%    & 2.72\%   & 2.88\%   \\ \cline{2-2} 
                              & 15\% Noise      & 27.39\% & 27.82\%  & 24.11\%   & 23.69\%  & 23.10\%  \\ \hline
    \hline
    \multicolumn{2}{|c|}{}&\multicolumn{5}{|c|}{SSIM} \\ \hline
    \hline
    \multirow{2}{*}{SVD}      & 1\% Noise       & 0.84    & 0.84     & 0.84       & 0.84     & 0.84
    \\ \cline{2-2} 
                              & 15\% Noise      & 0.61    & 0.62     & 0.68       & 0.63     & 0.64   \\ \hline
    \multirow{2}{*}{Wav}      & 1\% Noise       & 0.98    & 0.98     & 0.98       & 0.98     & 0.98   \\ \cline{2-2} 
                              & 15\% Noise      & 0.49    & 0.50     & 0.49       & 0.49     & 0.49   \\ \hline
    \multirow{2}{*}{Exp}      & 1\% Noise       & 0.97    & 0.97     & 0.97       & 0.97     & 0.97   \\ \cline{2-2} 
                              & 15\% Noise      & 0.48    & 0.57     & 0.59       & 0.51     & 0.51   \\ \hline
    
    \end{tabular}
    \caption{Error values for the numerical experiments with the Shepp-Logan phantom.}
    \label{table_errorapriori}
\end{table}

Figure~\ref{fig_svdapriori},~\ref{fig_Wavapriori}, and~\ref{fig_expapriori} depict several reconstruction results using the a-priori parameter choice rule $\alpha=0.5\,\delta$. The results presented in these figures are structured as follows: First column: FD without additional regularization (FD), second column: FD with Tikhonov regularization (FD-T), third column: FD with Landweber iteration (FD-L). Top row: $1\%$ relative noise, bottom row: $15\%$ relative noise. The corresponding error measures are collected in Table~\ref{table_errorapriori}, which also includes results for our a-posteriori parameter choice rule \eqref{discrepancy_FD}. However, since these results are visually very similar to those obtained with the a-priori parameter choice rule, we decided not to include them in this paper. Note that since in our experiments the frame functions $f_k$ are orthonormal, there holds $C_2=1$, and thus \eqref{tau} provides a computable lower bound for $\tau$ in our a-priori parameter choice rule \eqref{discrepancy_FD}. However, we empirically found the choice $\tau=0.1$ for the SVD case and $\tau=20$ for the FD cases to lead to much better results, and thus used them in all of the presented numerical experiments.

Comparing the obtained results as summarized in Table~\ref{table_errorapriori}, we see that for a relative noise level of $1\%$, additional regularization beyond the truncation inherent in the discretization is only beneficial for the SVD case. However, for a noise level of $15\%$ noise, additional regularization via the Tikhonov and Landweber filter functions is often beneficial, which can be seen from the error measures, in particular from the SSIM. It appears that the additional regularization has the most impact on the SVD, while having less impact on the exponential- and the wavelet-based FDs. This can be explained by the fact that for the chosen range of the indices $j$ we have $\alpha_j\in[0.71,0.94]$ for the wavelet-based FD and $\alpha_j\in[0.15,1]$ for the exponential-based FD. In contrast, the singular values $\sigma_m$ of the SVD lie within the interval $[0.27, 3.54]$ in our chosen range of indices $m$. However, note that since these index ranges were chosen such that in both the SVD and the exponential-based FD case the same number of singular/frame functions are used, we find that the FDs lead to more accurate reconstructions in the case of low noise levels than those obtained via the SVD, while the SVD shows better stability and regularization properties in the case of high noise levels, at a comparable computational cost.

\begin{figure}[ht!]
    \centering
    \includegraphics[width=0.8\textwidth]{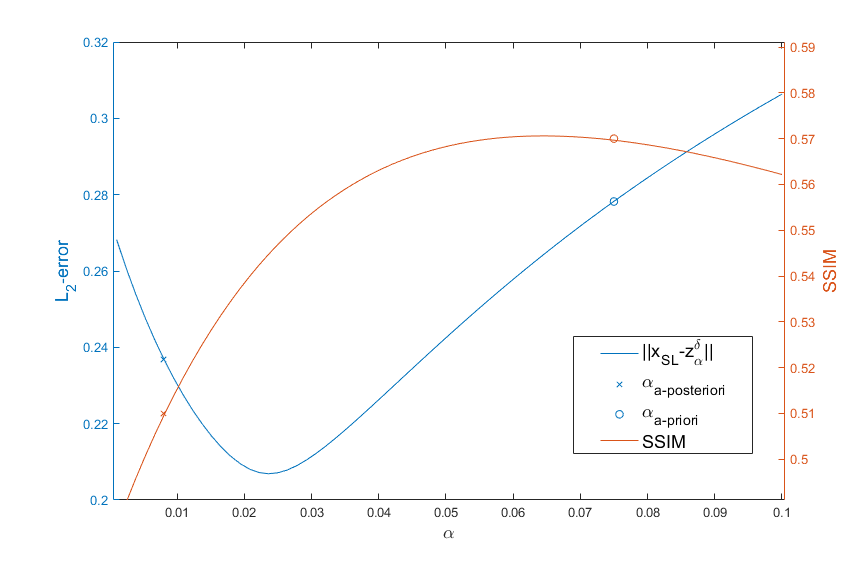}
    \caption{Error-plot for the exponential-based FD \eqref{num_FDexpnoise} with Tikhonov filter reconstructing the Shepp-Logan phantom with $15\%$ noise. The marks show the different parameter choices.The minimal possible error is obtained at $\alpha=0.023$ and yields an error of $20.69\%$. The maximal SSIM of 0.57 is obtained at $\alpha=0.061$}
    \label{fig_stoppruleserror}
\end{figure}

Finally, Figure~\ref{fig_stoppruleserror} illustrates the dependence of the error-measures on the regularization parameter. The marks indicate the parameters selected by the a-priori and a-posteriori parameter choice rules. In particular, note that the optimal (maximal) value of the SSIM is reached at a larger value of $\alpha$ than the optimal (minimal) $L_2$-error.

% % % % % % % % % % % % % % % % % % %
% Section - Conclusion and Outlook  %
% % % % % % % % % % % % % % % % % % %
\section{Conclusion}\label{sect_conclusion}

In this paper, we considered general continuous regularization methods based on FDs for linear ill-posed problems in Hilbert spaces. In particular, we proved convergence and convergence rates results under a-priori and a-posteriori parameter choice rules analogous to those for SVD-based regularization methods. Furthermore, we applied our results to a standard tomography problem based on the Radon transform, using specific FDs based on wavelets and exponential functions. The obtained results demonstrate that FDs are a viable approach for efficiently solving linear ill-posed problems.

% % % % % % % % % % %
% Section - Support %
% % % % % % % % % % %
\section{Support}

SH and RR were funded by the Austrian Science Fund (FWF): F6805-N36. LW was supported by the strategic program ``Innovatives O\"O 2010 plus" by the Upper Austrian Government and by the Austrian Science Fund (FWF): W1214-N15, project DK8

% % % % % % % % %
% Bibliography  %
% % % % % % % % %
\bibliographystyle{plain}
{\footnotesize
\bibliography{mybib}
}

% % % % % % % % % % % % % % % %
% Appendix - A small erratum  %
% % % % % % % % % % % % % % % %
\section*{Appendix: Minor Erratum}

\vspace{5pt}

\noindent
In this appendix, we want to correct two minor errors in our previous publication \cite{Hubmer_Ramlau_2021_01}. 

First of all, it was stated that the stability condition \eqref{cond_A_stability} implies that $A$ is continuously invertible as an operator from $X \to Z$. This is clearly wrong, since in general $R(A) \subsetneq Z$. However, it is true that condition \eqref{cond_A_stability} implies that for each $y \in R(A)$ there exists a unique solution $x = A^{-1} y \in X$ and that in this case $\norm{A^{-1}y}_X \leq (1/c_1)\norm{y}_Z$. Fortunately, this only marginally changes the results of the paper, such that in Theorem~4.3 and Theorem~4.8 
it should read $ y \in R(A)$ instead of $y \in Z$. Consequently, in Theorem~5.1, 5.2, 5.3 it should then also read $y \in R(A)$ instead of $y \in H^{\alpha+1/2}(\OS)$. 

Secondly, the coefficients in the FDs of the Radon transform given in Theorem~5.2 and Remark~5.2 are incorrect. However, the proper coefficients can be derived from Theorem~5.1 and are given correctly in this paper (see Theorem~\ref{thm_Radon_wavelets} and Theorem~\ref{thm_Radon_exp}).

\end{document}